\documentclass[11pt,a4paper]{article}
\usepackage{amsmath}
\usepackage{amsthm}
\usepackage{amssymb}
\usepackage{color}
\usepackage{cancel}
\usepackage{graphicx}
\usepackage[authoryear]{natbib}
\usepackage[unicode=true,
 bookmarks=false,
 breaklinks=false,pdfborder={0 0 1},backref=section,colorlinks=false]
 {hyperref}

\makeatletter

\providecommand{\tabularnewline}{\\}

\theoremstyle{plain}
\newtheorem{thm}{\protect\theoremname}
\theoremstyle{plain}
\newtheorem{lem}[thm]{\protect\lemmaname}
\theoremstyle{plain}
\newtheorem{prop}[thm]{\protect\propositionname}
\theoremstyle{plain}
\newtheorem{cor}[thm]{\protect\corollaryname}

\usepackage{color}
\usepackage{here}\usepackage{setspace}\usepackage{multirow}\usepackage{amsfonts}\usepackage{enumitem}\usepackage{placeins}\usepackage{mathrsfs}\usepackage{ragged2e}
\usepackage{subfigure}

\newcommand{\blind}{1}

\usepackage[dvipsnames]{xcolor}

\bibpunct{(}{)}{;}{a}{}{,} 

\def\U{\mbox{\boldmath$U$}}

\def\I{\mbox{\boldmath$I$}}
\def\zero{\mbox{\boldmath$0$}}
\def\one{\mbox{\boldmath$1$}}
\def\x{\mbox{\boldmath$x$}}
\def\y{\mbox{\boldmath$y$}}

\def\bbeta{\mbox{\boldmath$\beta$}}
\def\bPhi{\mbox{\boldmath$\Phi$}}

\def\btheta{\mbox{\boldmath$\theta$}}

\def\u{\mbox{\boldmath$u$}}

\def\bomega{\mbox{\boldmath$\omega$}}

\def\m{\mbox{\boldmath$m$}}

\def\mA{\mathcal{A}}

\def\mH{\mathcal{H}}
\def\mM{\mathcal{M}}
\def\seq#1#2{#1{:}#2}

\theoremstyle{definition}

\newtheorem{theorem}{Theorem}[section]

\newtheorem{remark}{Remark}[section]

\providecommand{\corollaryname}{Corollary}
\providecommand{\lemmaname}{Lemma}
\providecommand{\propositionname}{Proposition}
\providecommand{\theoremname}{Theorem}

\AtBeginEnvironment{thebibliography}{\linespread{1}\selectfont}

\makeatother

\providecommand{\corollaryname}{Corollary}
\providecommand{\lemmaname}{Lemma}
\providecommand{\propositionname}{Proposition}
\providecommand{\theoremname}{Theorem}

\begin{document}
\global\long\def\spacingset#1{ \global\long
\global\long\def\baselinestretch{%
}%
\small\normalsize}%
\spacingset{1}


\if1\blind { 
\title{ Equivariant online predictions of non-stationary time series}
\author{Kōsaku Takanashi\thanks{RIKEN Center for Advanced Intelligence Project, Tokyo, Japan. {\scriptsize{}{}Email:
kosaku.takanashi@riken.jp } } \, \& Kenichiro McAlinn\thanks{Corresponding author. Department of Statistics, Operations, and Data
Science, Fox School of Business, Temple University, Philadelphia,
PA 19122. {\scriptsize{}{}Email: kenichiro.mcalinn@temple.edu} }}

\maketitle
} \fi

\if0\blind { \bigskip{}
 \bigskip{}
 \bigskip{}

\begin{center}
{\LARGE{}{}Equivariant online predictions of non-stationary time
series} 
\par\end{center}

\medskip{}
 } \fi

\bigskip{}

\begin{abstract}
We discuss the finite sample theoretical properties of online predictions
in non-stationary time series under model misspecification. To analyze the theoretical predictive
properties of statistical methods under this setting, we first define the Kullback-Leibler risk, in order to place the problem within a decision theoretic framework. Under this framework, we show that a specific class of dynamic
models-- random walk dynamic linear models-- produce exact minimax
predictive densities. We first show this result under Gaussian assumptions,
then relax this assumption using semi-martingale processes. This result
provides a theoretical baseline, under both non-stationary and stationary
time series data, for which other models can be compared against.
We extend the result to the synthesis of multiple predictive densities.
Three topical applications in epidemiology, climatology, and economics,
confirm and highlight our theoretical results. 
\end{abstract}
 \vfill{}

\spacingset{1} 

\section{Introduction\label{sec:intro}}

Prediction and decision making using time series data are central
tasks for statistical methods in many fields, including economics,
climatology, and epidemiology. For these problems, the goal of the
statistical model is to use past data to characterize future data
and its uncertainty in an online, sequential manner. Predictions are then
used to inform and improve decision making. As these tasks are defined by their
temporality-- past events and actions affect future events-- the
success of a statistical model depends largely on its ability to capture
temporal characteristics. To respond, many models have been developed
and proposed, often categorized as dynamic models \citep[see, e.g.,][]{WestHarrison1997book2,Prado2010}.
Although the success of these models has been shown in practice,
theoretical investigations have been limited, particularly for non-stationary
data, despite its relevance in the domain.

There are several reasons why theoretical development in this area
has been difficult, although the biggest reason is inherent in the
data itself. As the data in time series tasks are, prima facie, non-stationary,
defined by the lack of stationary assumptions, most of the theoretical
apparatus used in statistics, which rely on assumptions, such as i.i.d.,
cannot be used. For example, in most time series data, we observe
gradual and sudden shifts in trend and volatility as time progresses, and a model or covariate
that works at one period often do not in another. This problem is compounded
when we consider decisions, as these decisions themselves can affect
future data (e.g. implementing strict mobility restrictions under
a pandemic will likely affect future infection numbers). To partially
get around this issue, many papers transform the data to appear stationary,
often by taking the log difference several times. This, however, is
often detrimental to the task itself, as these transformations do
not guarantee that the data are actually stationary and, more importantly,
they remove  the signals and patterns that are critical to understand
the data and make better predictions and decisions.

In developing theory for time series models, three points
must be considered for it to be relevant to real world problems.
First, given the non-stationary nature of the data, it cannot rely
on asymptotics, nor the conditions required for it. This is not only
required by the temporal characteristics of the data, but also reflects
the need of the decision maker, since forecasting and decision making
are always done with finite sample data and for finite horizons. For
example, many economic data are measured at the monthly frequency,
with only several decades in the past being relevant, at most (and
much less so under economic shocks). This would not be enough data
for the asymptotic results to be meaningful, even if the data can
be assumed to be i.i.d., which they cannot. Second, since decision
making is done under uncertainty, the theoretical results must be
with regard to the whole predictive distribution, and not just the
mean. Although the predictive literature has often focused on point metrics, such as mean squared forecast error, this is often insufficient
for decision making. For example, in finance, investment portfolios
are constructed by taking into account both the mean and covariance
of the assets, and it is widely recognized that the latter quantification
is critical for successful asset management. Third, all models must
be assumed to be misspecified, a setting often referred to as an $\mM$-open 
\citep[following][]{bernardo2009bayesian}. This point reflects
the reality that the ``best" model at some point can change given
time or even under some shock. It is simply unrealistic to assume
that there exists a true model in time series contexts. Given all
three points, thus, theoretical evaluations of statistical models
for non-stationary time series must be done in finite sample, with
regard to its distributional predictions, and where a true model cannot
be assumed to exist, for it to be meaningful and provide insight for
practice.

We contribute to this field by developing a theoretical framework
that satisfies all three criteria. Specifically, we have three contributions.
First, we define the Kullback-Leibler risk for predictions in non-stationary
time series data (Section~\ref{sec:KL}). This allows us to conduct
theoretical analyses of statistical models for non-stationary time
series data within a decision theoretic framework. Critical to this
framework, we neither assume that the true model is nested nor any
asymptotics. Second, using this decision theoretic framework, we show
that a specific class of dynamic models produces exact minimax predictive
distributions, under Gaussian assumptions (Section~\ref{subsec:maintheorem}).
This result provides finite sample predictive guarantees, making it
a benchmark to compare other models against. We then generalize this
result by relaxing the Gaussian assumption using semi-martingale processes,
providing theoretical guarantees under more general settings (Section~\ref{App:semimart}).
Third, we extend the above results to the problem of combining several
predictive distributions (i.e. ensemble methods), and show that dynamic
Bayesian predictive synthesis \citep[BPS:][]{mcalinn2019dynamic}
is exact minimax (Section~\ref{sec:bps}). We highlight the theoretical
results using three topical datasets: weekly average COVID-19 cases
in Tokyo, monthly global mean sea level, and a high-dimensional monthly
economic dataset (Section~\ref{sec:emst}). Through all three applications
(and simulation results in Appendix~\ref{sec:sim}), we show that
a method that achieves exact minimaxity is superior to methods that
do not.

\section{Preliminaries\label{sec:setup}}

\subsection{Data generating process}

For the theoretical analysis, we first assume that the data generating
process of both the target and covariates
follow Gaussian processes. While this assumption is somewhat restrictive--
even if it is reasonable enough for most applications--, this is
primarily done for ease of exposition, and to clarify the assumptions
we make. This assumption of Gaussianity is relaxed and generalized
using semi-martingale processes in Section~\ref{App:semimart}.

Let the $\mathbb{R}$-valued process to be predicted, $\left\{ y_{t}\right\} $,
and the covariate $\mathbb{R}$-valued processes (which can include
the lag of $y$), $\left\{ x_{j,t}\right\} $, be Gaussian processes.
Thus, any finite linear combination of elements from a set of random
variables, $\left\{ y_{s},x_{j,s};s=1,\cdots t,\ j\in J\right\} $,
is a Gaussian random variable. Let the mean and variance of $\left\{ y_{t}\right\} $
be $\mu_{t}^{y}$ and $\sigma_{t}^{y}$, respectively, and (the elements
of) the mean vector and covariance matrix of $\left\{ x_{j,t}\right\} $
be 
\[
\mu_{t}^{j}=\mathbb{E}\left[x_{j,t}\right],\ \sigma_{t}^{ij}=\mathbb{E}\left[\left(x_{j,t}-\mu_{t}^{j}\right)\left(x_{i,t}-\mu_{t}^{i}\right)\right].
\]

Let the $J+1$-dimensional multivariate normal distribution of $\left(y_{t},\boldsymbol{x}_{t}\right)$
be 
\[
\left[\begin{array}{c}
y_{t}\\
\boldsymbol{x}_{t}
\end{array}\right]\sim N\left(\left[\begin{array}{c}
\mu_{t}^{y}\\
\boldsymbol{\mu}_{t}
\end{array}\right],\left[\begin{array}{cc}
\sigma_{t}^{y} & \boldsymbol{\sigma}_{t}^{yJ\top}\\
\boldsymbol{\sigma}_{t}^{yJ} & \varSigma_{t}^{J}
\end{array}\right]\right),
\]
where $\boldsymbol{x}_{t}=\left[x_{1,t},\cdots,x_{J,t}\right]^{\top}$
$\boldsymbol{\mu}_{t}=\left[\mu_{t}^{1},\cdots,\mu_{t}^{J}\right]^{\top}$,
$\boldsymbol{\sigma}_{t}^{yJ}=\left[\sigma_{t}^{y1},\cdots,\sigma_{t}^{yJ}\right]^{\top}$,
$\sigma_{t}^{yj}=\mathbb{E}\left[\left(y_{t}-\mu_{t}^{y}\right)\left(x_{j,t}-\mu_{t}^{j}\right)\right]$,
and $\varSigma_{t}^{J}$, which is a $J\times J$ matrix with $\sigma_{t}^{ij}$
as the $i,j$-th element. Then, the conditional distribution, $p_{t}\left(y_{t}\left|\boldsymbol{x}_{t}\right.\right)$,
is, 
\[
p_{t}\left(y_{t}\left|\boldsymbol{x}_{t}\right.\right)=N\left(\mu_{t}+\left(\boldsymbol{\sigma}_{t}^{yJ}\right)^{\top}\left(\varSigma_{t}^{J}\right)^{-1}\left(x_{t}-\mu_{t}\right),\sigma_{t}^{y}-\left(\boldsymbol{\sigma}_{t}^{yJ}\right)^{\top}\left(\varSigma_{t}^{J}\right)^{-1}\boldsymbol{\sigma}_{t}^{yJ}\right),
\]
from the normal correlation theorem. If we write, 
\[
\mu_{t}+\left(\boldsymbol{\sigma}_{t}^{yJ}\right)^{\top}\left(\varSigma_{t}^{J}\right)^{-1}\left(\boldsymbol{x}_{t}-\mu_{t}\right)=\theta_{0,t+1}^{*}+\left\langle \boldsymbol{\theta}_{t+1}^{*},\boldsymbol{x}_{t+1}\right\rangle, 
\]
where $\boldsymbol{\theta}_{t+1}^{*}=\left[\theta_{1,t+1}^{*},\cdots,\theta_{J,t+1}^{*}\right]^{\top}$,
and $\sigma_{t+1}=\sigma_{t}^{y}-\left(\boldsymbol{\sigma}_{t}^{yJ}\right)^{\top}\left(\varSigma_{t}^{J}\right)^{-1}\boldsymbol{\sigma}_{t}^{yJ}$,
we have the likelihood function of the best model, 
\begin{align}
p_{t}\left(y_{t+1}\left|\boldsymbol{x}_{t}\right.\right) & =p_{t}\left(y_{t+1}\left|\theta_{0,t+1}^{*},\boldsymbol{\theta}_{t+1}^{*},\sigma_{t+1},\boldsymbol{x}_{t+1}\right.\right)\nonumber \\
 & =\frac{1}{\sqrt{2\pi\sigma_{t+1}^{2}}}\exp\left(-\frac{\left(y_{t+1}-\theta_{0,t+1}^{*}-\left\langle \boldsymbol{\theta}_{t+1}^{*},\boldsymbol{x}_{t+1}\right\rangle \right)^{2}}{2\sigma_{t+1}^{2}}\right).\label{eq:DGPconditional}
\end{align}

Given the above formulation, the unknown parameters are $\left(\theta_{0,t+1}^{*},\boldsymbol{\theta}_{t+1}^{*},\sigma_{t+1}\right)$.
By defining the best possible prediction, and identifying the unknown
parameters, we are able to measure and analyze the predictive risk
with regard to its Kullback-Leibler (KL) divergence. 
The specific definition of the KL loss and risk used in the statistical decision theory is given in Section~\ref{sec:KL}.

Note that our interest
is in the predictive ability of time series models, and not in the geometric
form of them. Because our interest is in the 1-step ahead forecasts,
the model space that defines the KL risk is solely determined by the
$J+2$ parameters, $\left(\theta_{0,t+1}^{*},\boldsymbol{\theta}_{t+1}^{*},\sigma_{t+1}\right)$,
and they need not be seen as a function of $t$.

\subsection{Equivariant model }

We consider an equivariant model for the decision problem. An equivariant
model is defined here as a model that produces predictive distributions
that are equivariant for the decision problem concerning
the one-step ahead predictive KL loss function. Specifically, the
equivariant model regarding the DGP (eq.~\ref{eq:DGPconditional})
is given as follows (the proof of its equivariance is given in \ref{subsec:Equivariance}):
\begin{subequations}
\label{DLM} 
\begin{align}
y_{t+1} & =\theta_{0,t+1}+\boldsymbol{x}_{t+1}^{\top}\boldsymbol{\theta}_{t+1}+\nu_{t+1},\quad\nu_{t+1}\sim N\left(0,v_{t+1}\right),\label{eq:DLMa}\\
\boldsymbol{\theta}_{t+1}^{\prime} & =\boldsymbol{\theta}_{t}^{\prime}+\boldsymbol{\omega}_{t+1},\quad\boldsymbol{\omega}_{t+1}\sim N\left(0,v_{t+1}\boldsymbol{W}_{t+1}\right),\label{eq:DLMb}\\
v_{t+1} & =\frac{\beta}{\gamma_{t+1}}v_{t},\quad\gamma_{t+1}\sim\textrm{Beta}\left(\frac{\beta n_{t}}{2},\frac{\left(1-\beta\right)n_{t}}{2}\right),\label{eq:DLMc}
\end{align}
\end{subequations}
 which is a random walk DLM, where $\boldsymbol{\theta}_{t+1}^{\prime}=\left[\theta_{0,t+1},\boldsymbol{\theta}_{t+1}\right]^{\top}$
evolves in time according to a linear, normal random walk with innovations
variance matrix, $v_{t+1}\boldsymbol{W}_{t+1}$, at time $t+1$, and
$v_{t+1}$ is the residual variance in predicting $y_{t+1}$ based
on past information and the set of covariates. The
residuals, $\nu_{t+1}$, and evolution innovations, $\boldsymbol{\omega}_{s+1}$,
are independent over time and mutually independent for all $t,s.$

Specifically, the likelihood function is 
\begin{align}
\hat{p}_{t}\left(y_{t+1}\left|\boldsymbol{\theta}_{t+1}^{\prime},v_{t+1},\boldsymbol{x}_{t+1}\right.\right) & =\frac{1}{\sqrt{2\pi v_{t+1}}}\exp\left(-\frac{\left(y_{t+1}-\left\langle \boldsymbol{\theta}_{t+1}^{\prime},\boldsymbol{x}_{t+1}^{\prime}\right\rangle \right)^{2}}{2v_{t+1}}\right),\label{eq:LikelihoodBPS}
\end{align}
where $\boldsymbol{x}_{t+1}^{\prime}=\left[1,\boldsymbol{x}_{t+1}\right]^{\top}$.
The predictive distribution conditioned on the covariates,
$\boldsymbol{x}_{t+1}^{\prime}$, are 
\begin{alignat}{1}
 & \hat{p}_t\left(y_{t+1}\left|\left\{ y_{s},\boldsymbol{x}_{s}\right\} _{s=1}^{t},\boldsymbol{x}_{t+1}\right.\right)\nonumber \\
= & \int_{\left(0,\infty\right)}\int_{\mathbb{R}^{J+1}}\hat{p}_{t}\left(y_{t+1}\left|\boldsymbol{\theta}_{t+1}^{\prime},\boldsymbol{x}_{t+1}^{\prime},v_{t+1}\right.\right)\pi\left(\boldsymbol{\theta}_{t+1}^{\prime},v_{t+1}\left|\left\{ y_{s},\boldsymbol{x}_{s}\right\} _{s=1}^{t}\right.\right)d\boldsymbol{\theta}_{t+1}^{\prime}dv_{t+1}.\label{eq:BPSconditional}
\end{alignat}
Here, 
\[
\pi\left(\boldsymbol{\theta}_{t+1}^{\prime},v_{t+1}\left|\left\{ y_{s},\boldsymbol{x}_{s}\right\} _{s=1}^{t}\right.\right)=\int_{\left(0,\infty\right)}\int_{\mathbb{R}^{J+1}}\pi\left(\boldsymbol{\theta}_{t+1}^{\prime},v_{t+1}\left|\boldsymbol{\theta}_{t}^{\prime},v_{t}\right.\right)\pi\left(\boldsymbol{\theta}_{t}^{\prime},v_{t}\left|\left\{ y_{s},\boldsymbol{x}_{s}\right\} _{s=1}^{t}\right.\right)d\boldsymbol{\theta}_{t}^{\prime}dv_{t},
\]
where $\pi\left(\boldsymbol{\theta}_{t}^{\prime},v_{t}\left|\left\{ y_{s},\boldsymbol{x}_{s}\right\} _{s=1}^{t}\right.\right)$
is the posterior distribution of $\left(\boldsymbol{\theta}_{t}^{\prime},v_{t}\right)$
at time $t$, after observing $y_{t}$ and $\boldsymbol{x}_{t+1}$.
At $t=0$, we denote the (initial) prior distribution of $\left(\boldsymbol{\theta}_{0}^{\prime},v_{0}\right)$
as $\rho\left(\boldsymbol{\theta}_{0}^{\prime},v_{0}\right)$. Under
this (initial) prior, the predictive distribution is denoted as $\hat{p}_{t}^{\rho}\left(y_{t+1}\left|\left\{ y_{s},\boldsymbol{x}_{s}\right\} _{s=1}^{t},\boldsymbol{x}_{t+1}\right.\right)$,
and the posterior distribution as $\pi^{\rho}\left(\boldsymbol{\theta}_{t}^{\prime},v_{t}\left|\left\{ y_{s},\boldsymbol{x}_{s}\right\} _{s=1}^{t}\right.\right)$,
both at time $t$.

Note that we can correspond eq.~\eqref{eq:DGPconditional} to eq.~\eqref{eq:LikelihoodBPS},
as follows: 
\begin{align*}
\theta_{0,t+1}^{*} & \leftrightarrow\theta_{0,t+1},\quad\boldsymbol{\theta}_{t+1}^{*}\leftrightarrow\boldsymbol{\theta}_{t+1},\\
z_{t+1} & \leftrightarrow\varepsilon_{t+1},\quad\sigma_{t+1}\leftrightarrow\sqrt{v_{t+1}}.
\end{align*}

In the next section, we show that the equivariant model also gives
a minimax predictive density.

\section{Finite sample predictive properties \label{sec:minmax}}

\subsection{A set up for the statistical decision problem\label{sec:KL}}

The goal is to show that the random walk DLM (eq.~\ref{DLM}), denoted
as $\hat{p}_{t}\left(y_{t+1}\left|\left\{ y_{s},\boldsymbol{x}_{s}\right\} _{s=1}^{t},\boldsymbol{x}_{t+1}\right.\right)$,
is equivariant and minimax in terms of the KL risk with regard to
the transition probability, $p_{t}\left(y_{t+1}\left|\left\{ y_{s},\boldsymbol{x}_{s}\right\} _{s=1}^{t},\theta_{0,t+1}^{*},\boldsymbol{\theta}_{t+1}^{*},\sigma_{t+1},\boldsymbol{x}_{t+1}\right.\right)$.
Thus, for each time the covariate processes, $\left\{ \boldsymbol{x}_{t}\right\} $,
are given, we can construct the predictive distribution, $\hat{p}_{t}\left(y_{t+1}\left|\left\{ y_{s},\boldsymbol{x}_{s}\right\} _{s=1}^{t},\boldsymbol{x}_{t+1}\right.\right)$,
in a path-wise manner. We will show that the random walk DLM is a
minimax estimate regarding the parameters, $\left(\theta_{0,t+1}^{*},\boldsymbol{\theta}_{t+1}^{*},\sigma_{t+1}\right)$,
of the best model, $p_{t}\left(y_{t+1}\left|\left\{ y_{s},\boldsymbol{x}_{s}\right\} _{s=1}^{t},\theta_{0,t+1}^{*},\boldsymbol{\theta}_{t+1}^{*},\sigma_{t+1},\boldsymbol{x}_{t+1}\right.\right)$.

We first define the invariant decision problem.
As the loss function, we employ KL divergence, and define it as follows:
\begin{alignat*}{1}
 & \mathsf{KL}\left(p_{t}\left(y_{t+1}\left|\left\{ y_{s},\boldsymbol{x}_{s}\right\} _{s=1}^{t},\theta_{0,t+1}^{*},\boldsymbol{\theta}_{t+1}^{*},\sigma_{t+1},\boldsymbol{x}_{t+1}\right.\right)\left|\hat{p}_{t}\left(y_{t+1}\left|\left\{ y_{s},\boldsymbol{x}_{s}\right\} _{s=1}^{t},\boldsymbol{x}_{t+1}\right.\right)\right.\right)\\
= & \int_{\mathbb{R}}p_{t}\left(y_{t+1}\left|\left\{ y_{s},\boldsymbol{x}_{s}\right\} _{s=1}^{t},\theta_{0,t+1}^{*},\boldsymbol{\theta}_{t+1}^{*},\sigma_{t+1},\boldsymbol{x}_{t+1}\right.\right)\log\frac{p_{t}\left(y_{t+1}\left|\left\{ y_{s},\boldsymbol{x}_{s}\right\} _{s=1}^{t},\theta_{0,t+1}^{*},\boldsymbol{\theta}_{t+1}^{*},\sigma_{t+1},\boldsymbol{x}_{t+1}\right.\right)}{\hat{p}_{t}\left(y_{t+1}\left|\left\{ y_{s},\boldsymbol{x}_{s}\right\} _{s=1}^{t},\boldsymbol{x}_{t+1}\right.\right)}dy_{t+1}.
\end{alignat*}
Then, the KL risk is 
\begin{alignat}{1}
 & R_{\mathsf{KL}}\left(\left(\theta_{0,t+1}^{*},\boldsymbol{\theta}_{t+1}^{*},\sigma_{t+1}\right),\hat{p}_{t}\right)=\label{eq:KLrisk}\\
 & \int_{\mathbb{R}^{\otimes t}}\mathsf{KL}\left(p_{t}\left(y_{t+1}\left|\left\{ y_{s},\boldsymbol{x}_{s}\right\} _{s=1}^{t},\theta_{0,t+1}^{*},\boldsymbol{\theta}_{t+1}^{*},\sigma_{t+1},\boldsymbol{x}_{t+1}\right.\right)\left|\hat{p}_{t}\left(y_{t+1}\left|\left\{ y_{s},\boldsymbol{x}_{s}\right\} _{s=1}^{t},\boldsymbol{x}_{t+1}\right.\right)\right.\right)d\mu_{Y}\left(y_{t},\cdots,y_{0}\right).\nonumber 
\end{alignat}
Here, the probability measure, $\mu_{Y}\left(y_{t},\cdots,y_{0}\right)$,
is a cylindrical measure of the process, $\left\{ y_{t}\right\} $.
The sample space is the  past observations, $\left\{ y_{s}\right\} _{s=1}^{t}$.
As a transformation to the sample space, $\left\{ y_{s}\right\} _{s=1}^{t}$,
we add the shift transform value, $-\left(\frac{1}{\sigma_{t+1}}\theta_{0,t+1}^{*}+\left\langle \frac{1}{\sigma_{t+1}}\boldsymbol{\theta}_{t+1}^{*},\boldsymbol{x}_{t+1}\right\rangle \right)$,
and multiply each scale by $\frac{1}{\sigma_{t+1}}$, $\left(\frac{1}{\sigma_{t+1}}>0\right)$, to obtain, 
\begin{alignat*}{1}
y_{1} & \rightarrow\tilde{y}_{1}=\frac{1}{\sigma_{t+1}}y_{1}-\left(\frac{1}{\sigma_{t+1}}\theta_{0,t+1}^{*}+\left\langle \frac{1}{\sigma_{t+1}}\boldsymbol{\theta}_{t+1}^{*},\boldsymbol{x}_{t+1}\right\rangle \right)\\
 & \vdots\\
y_{t} & \rightarrow\tilde{y}_{t}=\frac{1}{\sigma_{t+1}}y_{t}-\left(\frac{1}{\sigma_{t+1}}\theta_{0,t+1}^{*}+\left\langle \frac{1}{\sigma_{t+1}}\boldsymbol{\theta}_{t+1}^{*},\boldsymbol{x}_{t+1}\right\rangle \right).
\end{alignat*}
This scale-shift transformation can transform the sample to an arbitrary value on $\mathbb{R}$.
The parameter space is $\left(\theta_{0,t+1},\boldsymbol{\theta}_{t+1},v_{t+1}\right)\textrm{ or }\left(\theta_{0,t+1}^{*},\boldsymbol{\theta}_{t+1}^{*},\sigma_{t+1}\right)$.
As with the sample, to transform the parameter, we scale-shift
transform the parameters in the DLM model as, 
\begin{alignat*}{1}
\boldsymbol{\theta}_{t+1} & \rightarrow\widetilde{\boldsymbol{\theta}}_{t+1}=\frac{1}{\sigma_{t+1}}\boldsymbol{\theta}_{t}-\frac{1}{\sigma_{t+1}}\boldsymbol{\theta}_{t}^{*},\\
\theta_{0,t+1} & \rightarrow\widetilde{\theta}_{0,t+1}=\frac{1}{\sigma_{t+1}}\theta_{0,t+1}-\frac{1}{\sigma_{t+1}}\theta_{0,t+1}^{*},\\
v_{t+1} & \rightarrow\widetilde{v}_{t}=\frac{v_{t}}{\sigma_{t+1}^{2}},
\end{alignat*}
or the parameters in the DGP as,
\begin{alignat*}{1}
\boldsymbol{\theta}_{t+1}^{*} & \rightarrow\widetilde{\boldsymbol{\theta}}_{t+1}^{*}=0,\\
\theta_{0,t+1}^{*} & \rightarrow\widetilde{\theta}_{0,t+1}^{*}=0,\\
\sigma_{t+1}^{*} & \rightarrow\widetilde{\sigma}_{t}^{*}=1.
\end{alignat*}
The decision space is the probability density function, $p\left(y_{t+1}\right)$, of $y_{t+1}$.
For
each time $t+1$, we consider the scale-shift transform, 
\[
y_{t+1}\rightarrow\tilde{y}_{t+1}=\frac{1}{\sigma_{t+1}}y_{t+1}-\left(\frac{1}{\sigma_{t+1}}\theta_{0,t+1}^{*}+\left\langle \frac{1}{\sigma_{t+1}}\boldsymbol{\theta}_{t+1}^{*},\boldsymbol{x}_{t+1}\right\rangle \right),
\]
for the target variable, $y_{t+1}$. The transformation of the probability
density is denoted as 
\[
p\left(y_{t+1}\right)\rightarrow p\left(\tilde{y}_{t+1}\right).
\]

Further, since
\[
\widetilde{y}_{t+1}=\widetilde{\theta}_{0,t+1}^{*}+\left\langle \widetilde{\boldsymbol{\theta}}_{t+1}^{*},\boldsymbol{x}_{t+1}\right\rangle +\frac{1}{\sigma_{t+1}}\varepsilon_{t+1},\ \frac{1}{\sigma_{t+1}}\varepsilon_{t+1}\sim N\left(0,\frac{v_{t+1}}{\sigma_{t+1}^{2}}\right),
\]
we have,
\[
p_{t}\left(y_{t+1}\left|\left\{ y_{s},\boldsymbol{x}_{s}\right\} _{s=1}^{t},\tilde{\theta}_{0,t+1}^{*},\tilde{\boldsymbol{\theta}}_{t+1}^{*},\tilde{\sigma}_{t+1},\boldsymbol{x}_{t+1}\right.\right)=p_{t}\left(\tilde{y}_{t+1}\left|\left\{ y_{s},\boldsymbol{x}_{s}\right\} _{s=1}^{t},\theta_{0,t+1}^{*},\boldsymbol{\theta}_{t+1}^{*},\sigma_{t+1},\boldsymbol{x}_{t+1}\right.\right).
\]
Then, the KL, as a loss function, satisfies invariance: 
\begin{alignat*}{1}
 & \mathsf{KL}\left(p_{t}\left(y_{t+1}\left|\left\{ y_{s},\boldsymbol{x}_{s}\right\} _{s=1}^{t},\tilde{\theta}_{0,t+1}^{*},\tilde{\boldsymbol{\theta}}_{t+1}^{*},\tilde{\sigma}_{t+1},\boldsymbol{x}_{t+1}\right.\right)\left|\hat{p}_{t}\left(\tilde{y}_{t+1}\left|\left\{ y_{s},\boldsymbol{x}_{s}\right\} _{s=1}^{t},\boldsymbol{x}_{t+1}\right.\right)\right.\right)\\
= & \mathsf{KL}\left(p_{t}\left(y_{t+1}\left|\left\{ y_{s},\boldsymbol{x}_{s}\right\} _{s=1}^{t},\theta_{0,t+1}^{*},\boldsymbol{\theta}_{t+1}^{*},\sigma_{t+1},\boldsymbol{x}_{t+1}\right.\right)\left|\hat{p}_{t}\left(y_{t+1}\left|\left\{ y_{s},\boldsymbol{x}_{s}\right\} _{s=1}^{t},\boldsymbol{x}_{t+1}\right.\right)\right.\right).
\end{alignat*}
The data transformation, parameter transformation, and loss invariance define the invariant decision problem.

The predictive distribution, $\hat{p}_{t}$, that achieves 
\[
\hat{p}_{t}\left(\tilde{y}_{t+1}\left|\left\{ \tilde{y}_{s},\boldsymbol{x}_{s}\right\} _{s=1}^{t},\boldsymbol{x}_{t+1}\right.\right)=\hat{p}_{t}\left(y_{t+1}\left|\left\{ y_{s},\boldsymbol{x}_{s}\right\} _{s=1}^{t},\boldsymbol{x}_{t+1}\right.\right)
\]
is called an equivariant predictive distribution, and is minimax in terms of the KL risk, which minimizes the max risk:
\[
\max_{\theta_{0,t+1}^{*},\theta_{t+1}^{*},\sigma_{t+1}}R_{\mathsf{KL}}\left(\left(\theta_{0,t+1}^{*},\boldsymbol{\theta}_{t+1}^{*},\sigma_{t+1}\right),\hat{p}_{t}\right),\ \left(\textrm{for each }t\right).
\]
 This means that the sequentially updated
Bayes predictive distribution, $\hat{p}_{t}$, is sequentially minimax.
We will now show that the predictive distribution of eq.~\eqref{DLM}
is equivariant and minimax. 

The choice of criteria-- finite sample predictive minimax risk--
is done to reflect the need and reality of predictive tasks in non-stationary,
time series data. For example, considering the risk of the estimated
parameters, and not the predictive risk, is infeasible because the
parameters dynamically change as the data generating process itself
changes. Assuming strong conditions on the data generating process,
such that it is equivalent to i.i.d. data, defeats the purpose and
uninformative for the applications of interest. Further, any asymptotic
analysis will require similarly strong assumptions, even if the interest
is in predictive risk, making the analysis irrelevant. Evaluating
the predictive minimax risk against the best  predictive density
allows us to assess the predictive performance in finite sample, making
it possible to use it for non-stationary time series data and be relevant
to the task.

\subsection{Equivariance\label{subsec:Equivariance}}

As we will see later in Lemma~\ref{lemma:Invariant}, if $\left\{ \theta_{0,t},\boldsymbol{\theta}_{t},v_{t}\right\} $
is updated via a random walk, the transition distribution of the scale-shift
transformed $\left\{ \theta_{0,t},\boldsymbol{\theta}_{t},v_{t}\right\} $
is invariant. Thus, the statistical decision theoretic problem is
invariant with regard to the scale-shift transform. The predictive distribution
is given as 
\begin{alignat*}{1}
 & \hat{p}_{t}\left(\widetilde{y}_{t+1}\left|\left\{ \widetilde{y}_{s},\boldsymbol{x}_{s}\right\} _{s=1}^{t},\boldsymbol{x}_{t+1}\right.\right)\\
= & \int\int\hat{p}_{t}\left(\widetilde{y}_{t+1}\left|\widetilde{\boldsymbol{\theta}^{\prime}}_{t+1},\widetilde{v}_{t+1},\boldsymbol{x}_{t+1}\right.\right)\pi\left(\widetilde{\boldsymbol{\theta}^{\prime}}_{t+1},\widetilde{v}_{t+1}\left|\left\{ \widetilde{y}_{s},\boldsymbol{x}_{s}\right\} _{s=1}^{t}\right.\right)d\widetilde{\boldsymbol{\theta}^{\prime}}_{t+1}d\widetilde{v}_{t+1}.
\end{alignat*}
Here, $\left\{ \widetilde{\boldsymbol{\theta}}_{t-1},\cdots,\widetilde{\boldsymbol{\theta}}_{1}\right\} $ is $\left\{ \boldsymbol{\theta}_{t-1},\cdots,\boldsymbol{\theta}_{1}\right\} $
transformed by adding $-\frac{1}{\sigma_{t+1}}\boldsymbol{\theta}_{t}^{*}$
and multiplying $\frac{1}{\sigma_{t+1}}$.

Here, the prior distribution of $\boldsymbol{\theta}_{t}^{\prime}$,
at $t=0$, is a Lebesgue measure, $\mathsf{m}\left(\cdot\right)$,
the prior distribution of $v_{t}$ at $t=0$ is a scale-invariant
prior, $\frac{1}{v_{0}}$, and the state space evolves following a
random walk (eqs.~\ref{eq:DLMb} and \ref{eq:DLMc}). Then, the following
lemma holds. 
\begin{lem}
\label{lemma:Invariant} Consider a state space, $\boldsymbol{\theta}_{t}^{\prime},v_{t}$,
that is updated through a random walk (eqs.~\ref{eq:DLMb} and \ref{eq:DLMc})
evolution, where the prior distribution for $\boldsymbol{\theta}_{t}^{\prime}$
at $t=0$ is a Lebesgue measure, $\mathsf{m}\left(\cdot\right)$ and
$v_{t}$ at $t=0$ is a scale-invariant Jeffreys' prior $\frac{1}{v_{0}}$.
The state space, $\boldsymbol{\theta}_{t}^{\prime}$, is equivalent
in law with regard to the scale-shift transform, $\boldsymbol{\theta}_{t}^{\prime}\rightarrow\widetilde{\boldsymbol{\theta}^{\prime}}_{t}$,
\begin{eqnarray*}
\mu_{\theta}\left(\boldsymbol{\theta}_{0}^{\prime},\cdots,\boldsymbol{\theta}_{t}^{\prime}\right) & = & \mu_{\tilde{\theta}}\left(\widetilde{\boldsymbol{\theta}_{0}^{\prime}},\cdots,\widetilde{\boldsymbol{\theta}_{t}^{\prime}}\right),
\end{eqnarray*}
where each $\mu_{\theta}\left(\boldsymbol{\theta}_{0}^{\prime},\cdots,\boldsymbol{\theta}_{t}^{\prime}\right),\mu_{\tilde{\theta}}\left(\widetilde{\boldsymbol{\theta}_{0}^{\prime}},\cdots,\widetilde{\boldsymbol{\theta}_{t}^{\prime}}\right)$
is a cylindrical measure of state process $\left\{ \boldsymbol{\theta}_{t}^{\prime}\right\} $
and $\left\{ \widetilde{\boldsymbol{\theta}_{t}^{\prime}}\right\} $.
Similarly for, $v_{t}$, the cylindrical measure is invariant to the
scale transform, $v_{t}\rightarrow\frac{1}{\sigma_{t}}v_{t}=\tilde{v}_{t}$:
\begin{eqnarray*}
\mu_{v}\left(v_{0},\cdots,v_{t}\right) & = & \mu_{\tilde{v}}\left(\widetilde{v_{0}},\cdots,\widetilde{v_{t}}\right).
\end{eqnarray*}
\end{lem}

\begin{proof}
See Appendix~\ref{sec:lemMinmax}. 
\end{proof}
Since $\hat{p}_{t}\left(\widetilde{y}_{t+1}\left|\widetilde{\boldsymbol{\theta}^{\prime}}_{t+1},\widetilde{v}_{t+1},\boldsymbol{x}_{t+1}\right.\right)$
is given from eq.~\eqref{eq:LikelihoodBPS}, we have 
\[
\hat{p}_{t}\left(\widetilde{y}_{t+1}\left|\widetilde{\boldsymbol{\theta}^{\prime}}_{t+1},\widetilde{v}_{t+1},\boldsymbol{x}_{t+1}\right.\right)=\hat{p}_{t}\left(y_{t+1}\left|\boldsymbol{\theta}_{t+1}^{\prime},v_{t+1},\boldsymbol{x}_{t+1}\right.\right).
\]
Therefore, with Lemma~\ref{lemma:Invariant}, we have 
\[
\pi\left(\widetilde{\boldsymbol{\theta}^{\prime}}_{t+1},\widetilde{v}_{t+1}\left|\left\{ \widetilde{y}_{s},\boldsymbol{x}_{s}\right\} _{s=1}^{t}\right.\right)=\pi\left(\boldsymbol{\theta}_{t+1}^{\prime},v_{t+1}\left|\left\{ y_{s},\boldsymbol{x}_{s}\right\} _{s=1}^{t}\right.\right)
\]
and, as a result, can see that this is equivariant: 
\[
\hat{p}_{t}\left(\widetilde{y}_{t+1}\left|\left\{ \widetilde{y}_{s},\boldsymbol{x}_{s}\right\} _{s=1}^{t},\boldsymbol{x}_{t+1}\right.\right)=\hat{p}_{t}\left(y_{t+1}\left|\left\{ y_{s},\boldsymbol{x}_{s}\right\} _{s=1}^{t},\boldsymbol{x}_{t+1}\right.\right).
\]

\subsection{Exact minimaxity\label{subsec:maintheorem}}

Next, we will show that the equivariance predictive distribution,
$\hat{p}_{t}\left(y_{t+1}\left|\left\{ y_{s},\boldsymbol{x}_{s}\right\} _{s=1}^{t},\boldsymbol{x}_{t+1}\right.\right)$,
is minimax in finite sample.

Given the KL risk, we have the following theorem: \begin{theorem}\label{thm:Minimax}

\textit{Let the process to be predicted, $\left\{ y_{t}\right\} $,
and the covariate processes, $\left\{ x_{j,t}\right\} $, be Gaussian
processes. Then, the random walk DLM predictive distribution, $\hat{p}_{t}^{\frac{1}{v_{0}}\mathsf{m}}\left(y_{t+1}\left|\left\{ y_{s},\boldsymbol{x}_{s}\right\} _{s=1}^{t},\boldsymbol{x}_{t+1}\right.\right)$,
with Lebesgue measure prior, $\mathsf{m}$, and Jefferys' prior, $\frac{1}{v_{0}}$,
for the initial prior ($t=0$), is exact minimax with regard to the
Kullback-Leibler risk (eq.~\ref{eq:KLrisk}): 
\[
R_{\mathsf{KL}}\left(\left(\theta_{0,t+1}^{*},\boldsymbol{\theta}_{t+1}^{*},\sigma_{t+1}\right),\hat{p}_{t}^{\frac{1}{v_{0}}\mathsf{m}}\right)=\min_{\hat{p}_{t}}\max_{\theta_{0,t+1}^{*},\boldsymbol{\theta}_{t+1}^{*},\sigma_{t+1}}R_{\mathsf{KL}}\left(\left(\theta_{0,t+1}^{*},\boldsymbol{\theta}_{t+1}^{*},\sigma_{t+1}\right),\hat{p}_{t}\right).
\]
} \end{theorem}

The full proof is given in Appendix~\ref{sec:Proof-of-Minimax}. The general strategy of the proof for Theorem.~\ref{thm:Minimax}
is done by proving the following two conditions, according to the
equalizer rule \citep[Chapter 5.3.2.]{berger1985statistical}, following
Corollary 1. of \citet{george2006improved}: 
\begin{enumerate}
\item The KL risk of $\hat{p}_{t}^{\frac{1}{v_{0}}\mathsf{m}}\left(y_{t+1}\left|\left\{ y_{s},\boldsymbol{x}_{s}\right\} _{s=1}^{t},\boldsymbol{x}_{t+1}\right.\right)$
is constant for all sets of parameters, $\left(\theta_{0,t+1}^{*},\boldsymbol{\theta}_{t+1}^{*},\sigma_{t+1}\right)$:
\[
R_{\mathsf{KL}}\left(\left(\theta_{0,t+1}^{*},\boldsymbol{\theta}_{t+1}^{*},\sigma_{t+1}\right),\hat{p}_{t}^{\frac{1}{v_{0}}\mathsf{m}}\right)=c,\ ^{\forall}\left(\theta_{0,t+1}^{*},\boldsymbol{\theta}_{t+1}^{*},\sigma_{t+1}\right).
\]
\item $\hat{p}_{t}^{\frac{1}{v_{0}}\mathsf{m}}\left(y_{t+1}\left|\left\{ y_{s},\boldsymbol{x}_{s}\right\} _{s=1}^{t},\boldsymbol{x}_{t+1}\right.\right)$
is extended Bayes. In other words, the Bayes risk is approximated
by the prior sequence, $\left\{ \rho_{k}\right\} $, where $\sqrt{2\pi\left|\sigma_{k}W_{0}\right|}\rho_{k}\left(\boldsymbol{\theta}_{0}^{\prime},v_{0}\right)\rightarrow\frac{1}{v_{0}}\mathsf{m}\left(\boldsymbol{\theta}_{0}^{\prime}\right)$
: 
\[
\lim_{k\rightarrow\infty}B\left(\rho_{k},\hat{p}_{t}^{\rho_{k}}\right)=B\left(\rho_{k},\hat{p}_{t}^{\frac{1}{v_{0}}\mathsf{m}}\right).
\]
\end{enumerate}
The Bayes risk is defined as 
\[
B\left(\rho_{k},\hat{p}_{t}^{\rho_{k}}\right)=\int R_{\mathsf{KL}}\left(\left(\theta_{0,t+1}^{*},\boldsymbol{\theta}_{t+1}^{*},\sigma_{t+1}\right),\hat{p}_{t}^{\rho_{k}}\right)\rho_{k}\left(\theta_{0,t+1}^{*},\boldsymbol{\theta}_{t+1}^{*},\sigma_{t+1}\right),
\]
which is written with regard to the posterior parameters, $\left(\theta_{0,t+1}^{*},\boldsymbol{\theta}_{t+1}^{*},\sigma_{t+1}\right)$,
and not the prior on $\boldsymbol{\theta}_{0}^{\prime},\sigma_{0}$.
However, this Bayes risk is effectively equivalent to the prior specification
of $\boldsymbol{\theta}_{0}^{\prime},\sigma_{0}$, due to recursive
updating. More specifically, the posterior distribution, $\pi\left(\boldsymbol{\theta}_{t}^{\prime},\sigma_{t}\left|\left\{ y_{s},\boldsymbol{x}_{s}\right\} _{s=1}^{t}\right.\right)$,
is given by the generalized Bayes formula for filtering (see Appendix~C
for notation), 
\[
\pi\left(\boldsymbol{\theta}_{t}^{\prime},\sigma_{t}\left|\left\{ y_{s},\boldsymbol{x}_{s}\right\} _{s=1}^{t}\right.\right)=\frac{\int_{\left(\mathbb{R}^{J}\right)^{t-1}}\beta_{t}\left(\boldsymbol{\theta}^{\prime},\boldsymbol{\sigma},y\right)d\mu_{\theta,\sigma}^{\rho}\left(\left(\boldsymbol{\theta}_{0}^{\prime},\sigma_{0}\right),\cdots,\left(\boldsymbol{\theta}_{t-1}^{\prime},\sigma_{t-1}\right)\right)}{\int_{\left(\mathbb{R}^{J}\right)^{t}}\beta_{t}\left(\boldsymbol{\theta}^{\prime},\boldsymbol{\sigma},y\right)d\mu_{\theta,\sigma}^{\rho}\left(\left(\boldsymbol{\theta}_{0}^{\prime},\sigma_{0}\right),\cdots,\left(\boldsymbol{\theta}_{t}^{\prime},\sigma_{t}\right)\right)},
\]
though because we assume a random walk (eq.~\ref{eq:RWstate}) for
the process, $\left\{ \boldsymbol{\theta}_{t}^{\prime},\sigma_{t}\right\} $,
the cylindrical measure, $\mu_{\theta,\sigma}^{\rho}\left(\left(\boldsymbol{\theta}_{0}^{\prime},\sigma_{0}\right),\cdots,\left(\boldsymbol{\theta}_{t}^{\prime},\sigma_{t}\right)\right)$,
is determined by the prior distribution of $\boldsymbol{\theta}_{0}^{\prime},\sigma_{0}$:
$\rho\left(\boldsymbol{\theta}_{0}^{\prime}\right),\rho\left(\sigma_{0}\right)$.

\subsection{Discussion}

Since the predictive distribution of the random walk DLM, $\hat{p}_{t}^{\frac{1}{v_{0}}\mathsf{m}}$,
is risk constant (Appendix~\ref{subsec:risk const}) and extended
Bayes (Appendix~\ref{subsec:extend}), from \citet{Liang-Barron_04}
Theorem 1., it is exact minimax with regard to the KL risk. This result
can be interpreted as follows: under the least favorable prior distribution,
the Bayes solution is minimax. However, if the
state stochastic process, $\left\{ \boldsymbol{\theta}_{t}^{\prime}\right\} $,
is stationary, even if the initial prior distribution is a Lebesgue
measure, the updated posterior distribution converges in law to a
proper probability measure, due to the individual ergodic theorem.
Under this setting, the predictive distribution cannot be minimax,
since it is not shift invariant. Therefore, for minimaxity, the state
stochastic process, $\left\{ \boldsymbol{\theta}_{t}^{\prime}\right\} $,
must not be stationary.

Minimaxity is useful as a theoretical benchmark because it examines the predictive under no prior information, a situation that both Bayesians and frequentists can agree to be relevant.
In a situation where someone has prior information that informs their model or prior, that model or prior should, at the very least, be equal to or improve over the minimax predictive distribution.
This is why, for non-stationary time series data, the minimax model is the theoretical benchmark for predictive performance.
In other words, a statistical model should at least achieve minimaxity for it to be considered worthwhile, particularly when a minimax model is known to exist.

\section{Extension using semi-martingale processes\label{App:semimart}}

While the theorem in Section~\ref{sec:minmax} assumes that the data
generating process and covariates follow Gaussian processes,
this assumption can be relaxed using semi-martingale processes.

We assume that the predictive process, $y_{t}$, and the covariate
process, $x_{j,t}$, $j\in J$, are both square integralable processes
and filtration $\mathcal{F}_{t}$-adapted ($\mathcal{F}_{t}=\sigma\left(\left\{ y_{s},\boldsymbol{x}_{s}\right\} _{s=1}^{t}\right)$).
Then, both $y_{t}$, $x_{j,t}$, $j\in J$ are uniquely decomposed,
from the Doob decomposition, as the following: 
\begin{align*}
y_{t}=y_{0}+M_{t}+V_{t}, & \quad M_{0}=V_{0}=0\\
x_{j,t}=x_{j,0}+M_{j,t}+V_{j,t}, & \quad M_{j,0}=V_{j,0}=0,\ j\in J
\end{align*}
where $y_{0},$ $\left\{ x_{j,0}\right\} _{j\in J}$ is a $\mathcal{F}_{0}$-measurable
stochastic variable, $M,$ $\left\{ M_{j}\right\} _{j\in J}$ is a
square-integrable martingale, and $V,$ $\left\{ V_{j}\right\} _{j\in J}$
is square-integrable and predictable, thus $\mathcal{F}_{t-1}$-adapted.
Further, the optimal approximate process of $y_{t+1}$, utilizing
the covariate process, $x_{j,t+1},j\in J$, can be represented as
\citep[F\"{o}llmer-Schweizer (FS) decomposition]{schweizer1995minimal}:
\begin{equation}
y_{t+1} =\theta_{y_{t+1}}^{*}+\sum_{s=1}^{t}\left\langle \boldsymbol{\theta}_{s+1}^{*},\varDelta\boldsymbol{x}_{s}\right\rangle +z_{t+1},\label{eq:FSmin}
\end{equation}
where $\left(\theta_{y_{t+1}}^{*},\left\{ \boldsymbol{\theta}_{s}^{*}\right\} _{1\leqq s\leqq t+1}\right)$
are the true parameters of the FS decomposition (due to $\boldsymbol{\theta}_{t+1}^{*}$
being determined at $t$) and $\varDelta\boldsymbol{x}_{s}=\boldsymbol{x}_{s+1}-\boldsymbol{x}_{s}$.
When $\left\{ \theta_{y_{t+1}}^{*}\right\} $ is constant, but $\left\{ \boldsymbol{\theta}_{s}^{*}\right\} _{1\leqq s\leqq t+1}$
vary, then $z_{t+1}$ is a martingale process, which is orthogonal
to the martingale term of $\boldsymbol{x}_{t}$. While it might seem
like an issue that the difference, $\varDelta x$, appears as a covariate
in the linear model, it is done to make the cross term orthogonal.
Using this expression, we can rewrite the best predictive model of
$y_{t+1}$ at $t$ as 
\[
y_{t+1}=y_{t}+\theta_{y_{t+1}}^{*}-\theta_{y_{t}}^{*}+\left\langle \boldsymbol{\theta}_{t+1}^{*},\varDelta\boldsymbol{x}_{t}\right\rangle +z_{t+1}-z_{t}.
\]
This can be seen by subtracting the FS decomposition at $t$, $y_{t}=\theta_{y_{t}}^{*}+\sum_{s=1}^{t-1}\left\langle \boldsymbol{\theta}_{s+1}^{*},\varDelta\boldsymbol{x}_{s}\right\rangle +z_{t}$,
from eq.~\eqref{eq:FSmin}.

\begin{remark} \textit{A martingale process, $z_{t}$, is orthogonal
to $\boldsymbol{M}_{t}$ (the martingale component of $\boldsymbol{x}_{t}$)
when $\boldsymbol{M}_{t}z_{t}$ is a martingale. Here, $\varDelta\boldsymbol{M}_{t}\varDelta z_{t}$,
$\varDelta\boldsymbol{M}_{t}z_{t}$, and $\boldsymbol{M}_{t}\varDelta z_{t}$
all have expectation zero.}

\end{remark}

\begin{remark} \textit{When $y_{t}$, $\boldsymbol{x}_{t}$, and
$z_{t}$ are all square-integrable martingales, $\left\langle \boldsymbol{\eta},\boldsymbol{M}_{t}\right\rangle $
are martingales with regard to $^{\forall}\boldsymbol{\eta}\in\mathbb{R}^{J}$,
and assume, 
\[
\mathbb{E}\left[\left.\left\langle \boldsymbol{\eta},\varDelta\boldsymbol{M}_{t}\right\rangle ^{2}\right|\left\{ y_{s},\boldsymbol{x}_{s}\right\} _{s=1}^{t}\right]=\left\langle \boldsymbol{\eta},\Sigma\left(t\right)\boldsymbol{\eta}\right\rangle ,
\]
where $\Sigma\left(t\right)$ is a $\left\{ y_{s},\boldsymbol{x}_{s}\right\} _{s=1}^{t}$-measurable
$\mathbb{R}^{J\times J}$ matrix value function. Then, 
\[
\mathbb{E}\left[z_{t+1}-z_{t}\left|\boldsymbol{x}_{t+1},\left\{ y_{s},\boldsymbol{x}_{s}\right\} _{s=1}^{t}\right.\right]=0.
\]
Therefore, we have 
\[
\mathbb{E}\left[y_{t+1}\left|\boldsymbol{x}_{t+1},\left\{ y_{s},\boldsymbol{x}_{s}\right\} _{s=1}^{t}\right.\right]=y_{t}+\theta_{y_{t+1}}^{*}-\theta_{y_{t}}^{*}+\left\langle \boldsymbol{\theta}_{t+1}^{*},\varDelta\boldsymbol{x}_{t}\right\rangle ,
\]
where eq.~\eqref{eq:FSmin} is the optimal approximate projection
to the joint plane of covariate processes, $x_{j,t+1}$, $j\in J$, and
the extra term $\theta_{y_{t+1}}^{*}-\theta_{y_{t}}^{*}$.}

\end{remark}

Since the process, $\left\{ z_{t}\right\} $, is a square-integrable
martingale, if we consider $\sigma_{t+1}^{2}$ to be the quadratic
variation of $z_{t+1}$, then $\sigma_{t+1}^{2}$ is the conditional
variance of $\varDelta z_{t}$: $\sigma_{t+1}^{2}=\mathbb{E}\left[\left.\left(\varDelta z_{t}\right)^{2}\right|\left\{ y_{s},\boldsymbol{x}_{s}\right\} _{s=1}^{t}\right]$.

Here, we add a new assumption on the distribution of $\varDelta z_{t}$:
the probability distribution, $p_{t,\varDelta z}$, for $\varDelta z_{t}$
only has the scale parameter, $\sigma_{t+1}$. In other words, if
we scale transform $\varDelta z_{t}\rightarrow c\varDelta z_{t}$
and $\sigma_{t+1}^{2}\rightarrow c^{2}\sigma_{t+1}^{2}$, $p_{t,\varDelta z}$
is scale-invariant: 
\begin{equation}
p_{t,c\varDelta z_{t}}\left(\left.\cdot\right|\sigma_{t+1}\right)=p_{t,\varDelta z_{t}}\left(\left.\cdot\right|c\sigma_{t+1}\right).\label{eq:scale-invariant}
\end{equation}
This assumption justifies the approximate normality of $p_{t,\varDelta z}$,
when the time interval, $\varDelta t$, is not large (i.e. we assume
that the 1-step ahead is not in the distant future).

Let us rewrite the extra term as $\theta_{0,t+1}^{*}=\theta_{y_{t+1}}^{*}-\theta_{y_{t}}^{*}$,
and represent the unknown parameter at $t$ with $\left(\theta_{0,t+1}^{*},\boldsymbol{\theta}_{t+1}^{*}\right)$.
Then, noting that $y_{t+1}=y_{t}+\varDelta y_{t}$, the probability
density function of the increment of $y_{t}$, can be written using
the probability density function of its martingale process, $\varDelta z_{t}=z_{t+1}-z_{t}$,
as 
\[
p_{t}\left(y_{t+1}\left|\left\{ y_{s},\boldsymbol{x}_{s}\right\} _{s=1}^{t},\theta_{0,t+1}^{*},\boldsymbol{\theta}_{t+1}^{*},\sigma_{t}^{2},\varDelta\boldsymbol{x}_{t}\right.\right)=p_{t,\varDelta z}\left(\frac{y_{t+1}-\left(y_{t}+\theta_{0,t+1}^{*}+\left\langle \boldsymbol{\theta}_{t+1}^{*},\varDelta\boldsymbol{x}_{t}\right\rangle \right)}{\sigma_{t}^{2}}\right).
\]

\begin{remark} \textit{The specification of $\varDelta z_{t}$ remains
unknown, though we assume $\varepsilon_{t+1}$ follows a normal distribution.
This implies that the minimaxity of the predictive distribution can
be improved, strictly speaking, as the formulation of $\varDelta z_{t}$
can be improved. However, since $\left\{ z_{t}\right\} $ is assumed
to be a square-integrable martingale, if the interval of $\varDelta z_{t}$
is sufficiently small, $\varDelta z_{t}$ is approximately normal
from Donsker's invariant principle. In the following, we do not consider
the formulation of $\varDelta z_{t}$.} \end{remark}

We consider some assumptions on $y_{t}$, in preparation for minimax
analysis. For the decomposition of $y_{t}$ (eq.~\ref{eq:FSmin}),
we assume the martingale component of $\boldsymbol{x}_{t}$ and the
increment of its orthogonal martingale process, $z_{t}$, are square-integrable
for all $t$. The probability distribution function of $z_{t}$ is
denoted as $p_{z_{t}}\left(\cdot\right)$, where $p_{z_{t}}\left(\cdot\right)$
can depend on $t$, given the subscript $t$, and further assume eq.~\eqref{eq:scale-invariant}.
Note that $p_{z_{t}}\left(\cdot\right)$ can be considered as a probability
density conditional on $\boldsymbol{x}_{t}$, from \citet{williams1991probability},
Section 9.5, Section 14.14.

\begin{theorem}

\textit{Suppose the following conditions hold:}

\textit{(i) The predictive process, $y_{t}$, and covariate processes,
$x_{j,t},j\in J$, are square-integrable processes, and are filtration
$\left\{ y_{s},\boldsymbol{x}_{s}\right\} _{s=1}^{t}$-adapted;}

\textit{(ii) For the martingale component $\boldsymbol{M}_{t}$ of
$\boldsymbol{x}_{t}$, $\left\langle \boldsymbol{\eta},\boldsymbol{M}_{t}\right\rangle $
are martingales with regard to $^{\forall}\boldsymbol{\eta}\in\mathbb{R}^{J}$,
and assume there exists a $\left\{ y_{s},\boldsymbol{x}_{s}\right\} _{s=1}^{t}$-measurable
$\mathbb{R}^{J\times J}$ matrix value function, $\Sigma\left(t\right)$,
and 
\[
\mathbb{E}\left[\left.\left\langle \boldsymbol{\eta},\varDelta\boldsymbol{M}_{t}\right\rangle ^{2}\right|\left\{ y_{s},\boldsymbol{x}_{s}\right\} _{s=1}^{t}\right]=\left\langle \boldsymbol{\eta},\Sigma\left(t\right)\boldsymbol{\eta}\right\rangle .
\]
Then, the random walk DLM predictive distribution, $\hat{p}_{t}^{\frac{1}{v_{0}}\mathsf{m}}\left(y_{t+1}\left|\left\{ y_{s},\boldsymbol{x}_{s}\right\} _{s=1}^{t}\right.\right)$,
with Lebesgue measure prior, $\mathsf{m}$, and Jefferys' prior, $\frac{1}{v_{0}}$,
for the initial prior ($t=0$), is exact minimax with regard to the
Kullback-Leibler risk (eq.~\ref{eq:KLrisk}): 
\[
R_{\mathsf{KL}}\left(\left(\theta_{0,t+1}^{*},\boldsymbol{\theta}_{t+1}^{*},\sigma_{t+1}\right),\hat{p}_{t}^{\frac{1}{v_{0}}\mathsf{m}}\right)=\min_{\hat{p}_{t}}\max_{\theta_{0,t+1}^{*},\theta_{t+1}^{*},\sigma_{t+1}}R_{\mathsf{KL}}\left(\left(\theta_{0,t+1}^{*},\boldsymbol{\theta}_{t+1}^{*},\sigma_{t+1}\right),\hat{p}_{t}\right).
\]
} \end{theorem}

For $\min_{\hat{p}_{t}}$ in the minimax theorem, we construct the
Bayes predictive distribution of the DLM model following Section~\ref{app:DLMFS}
by corresponding the DLM parameters to the unknown parameters, $\{\theta_{0,t+1}^{*},\boldsymbol{\theta}_{t+1}^{*},\sigma_{t+1}\}$.
This allows us to construct the minimum Bayes risk distribution.

Given the above specification, the proof for this theorem follows
in a similar manner.

\subsection{Corresponding the DLM parameters to the unknown parameters\label{app:DLMFS}}

While the corresponding of DLM parameters to the unknown parameters
of the DGP was straightforward for the case where we assumed Gaussian
processes, the semi-martingale case is less straightforward.

Consider the random walk DLM, 
\begin{subequations}
\label{eq:BPS} 
\begin{alignat}{2}
y_{t+1} & =\theta_{0,t+1}+\left\langle \boldsymbol{\theta}_{t+1},\boldsymbol{x}_{t+1}\right\rangle +\varepsilon_{t+1}, & \quad & \varepsilon_{t+1}\sim N\left(0,v_{t+1}\right),\label{eq:BPS_measure}\\
\boldsymbol{\theta}_{t+1}^{\prime} & =\boldsymbol{\theta}_{t}^{\prime}+\boldsymbol{\omega}_{t+1}, &  & \boldsymbol{\omega}_{t+1}\sim N\left(0,v_{t+1}\boldsymbol{W}_{t+1}\right),\label{eq:RWstate}
\end{alignat}
\end{subequations}
 where $\boldsymbol{\theta}_{t}^{\prime}$ is the vector $\left[\theta_{0,t},\boldsymbol{\theta}_{t}\right]^{\top}$,
with intercept, $\theta_{0,t}$, $\boldsymbol{x}_{t}^{\prime}$ is
$\left[1,\boldsymbol{x}_{t}\right]^{\top}$, and $\boldsymbol{\omega}_{t+1}$
is a $J+1$-dimensional vector of normal random numbers. We will now
correspond the observational equation (eq.~\ref{eq:BPS_measure})
to the FS decomposition in eq.~\eqref{eq:FSmin}. First, noting that,
\begin{align*}
\left\langle \boldsymbol{\theta}_{t+1},\boldsymbol{x}_{t+1}\right\rangle -\left\langle \boldsymbol{\theta}_{t},\boldsymbol{x}_{t}\right\rangle  & =\left\langle \Delta\boldsymbol{\theta}_{t},\boldsymbol{x}_{t}\right\rangle +\left\langle \boldsymbol{\theta}_{t+1},\Delta\boldsymbol{x}_{t}\right\rangle \\
\theta_{0,t+1} & =\theta_{0,t}+\omega_{0,t},
\end{align*}
we have 
\[
y_{t+1}=\theta_{0,t+1}+\left\langle \boldsymbol{\theta}_{t},\boldsymbol{x}_{t}\right\rangle +\left\langle \Delta\boldsymbol{\theta}_{t},\boldsymbol{x}_{t}\right\rangle +\left\langle \boldsymbol{\theta}_{t+1},\Delta\boldsymbol{x}_{t}\right\rangle +\varepsilon_{t+1},
\]
by setting 
\begin{align*}
\theta_{0,t+1} & =\mathbb{E}\left[\theta_{0,t}\left|\left\{ y_{s},\boldsymbol{x}_{s}\right\} _{s=1}^{t}\right.\right]-\left\langle \Delta\boldsymbol{\theta}_{t},\boldsymbol{x}_{t}\right\rangle ,\\
\mathbb{E}\left[\theta_{0,t}\left|\left\{ y_{s},\boldsymbol{x}_{s}\right\} _{s=1}^{t}\right.\right] & =\theta_{0,0}+y_{t}-\left\langle \boldsymbol{\theta}_{t},\boldsymbol{x}_{t}\right\rangle ,
\end{align*}
where $\theta_{0,0}$ is an initial value of $\left\{ \theta_{0,t}\right\} $.
From this, we have 
\begin{equation*}
\mathbb{E}\left[\theta_{0,t+1}\left|\left\{ y_{s},\boldsymbol{x}_{s}\right\} _{s=1}^{t}\right.\right]  =\theta_{0,0}+y_{t}-\left\langle \boldsymbol{\theta}_{t},\boldsymbol{x}_{t}\right\rangle ,
\end{equation*}
where $\theta_{0,t}$ follows a random walk. Therefore, eq.~(\ref{eq:BPS_measure})
is transformed to 
\begin{equation}
y_{t+1} =y_{t}+\theta_{0,0}+\left\langle \boldsymbol{\theta}_{t+1},\varDelta\boldsymbol{x}_{t},\right\rangle +\varepsilon_{t+1},\label{eq:FS_BPS}
\end{equation}
where the parameters considered in eq.~(\ref{eq:FSmin}), $\left(\theta_{0,t+1}^{*},\boldsymbol{\theta}_{t+1}^{*}\right)$,
correspond to $\left(\boldsymbol{\theta}_{t+1}^{\prime}\right)$ in
eq.~(\ref{eq:FS_BPS}) through 
\begin{align*}
\theta_{0,t+1}^{*} & \leftrightarrow\theta_{0,t+1}+\left\langle \boldsymbol{\theta}_{t},\boldsymbol{x}_{t}\right\rangle +\left\langle \Delta\boldsymbol{\theta}_{t},\boldsymbol{x}_{t}\right\rangle -y_{t},\\
\boldsymbol{\theta}_{t+1}^{*} & \leftrightarrow\boldsymbol{\theta}_{t+1}.
\end{align*}
Further, we assume the error, $\varepsilon_{t+1}$, in eq.~(\ref{eq:FS_BPS}),
follows a normal distribution. In general, the distribution of $\varepsilon_{t+1}$
and $z_{t+1}$, obtained from the data generating process, need not
be the same. Since $z_{t+1}$ is a process that includes omitted variables,
it is unrealistic to exactly identify it.

\section{Predictive density synthesis \label{sec:bps}}

As an extension and application of the above theoretical results,
we consider the problem of ensembling multiple predictive distributions.
In time series contexts, model misspecification and uncertainty pose
a significant problem, particularly for predictive tasks. Ensemble
methods, which include model averaging and forecast combination, have
been used extensively to mitigate this uncertainty. In the Bayesian
context, some popular methods include Bayesian model averaging \citep[BMA:][]{hoeting1999bayesian}
and Bayesian stacking \citep{yao2018using}. Recent interest in combining
distributional information to account for better uncertainty quantification
has spurred development in ensembling predictive densities \citep[e.g.][]{HallMitchell2007,Geweke2011,Aastveit2014,diebold2019machine}.

While ensembling strategies have shown to be effective in practice,
improving over individual forecasts, the theoretical properties of
the ensembled predictions have not been investigated in a non-stationary
time series context. We are specifically interested in the Bayesian
predictive synthesis framework proposed in \citet{mcalinn2019dynamic},
which extended the work on expert opinion analysis \citep{GenestSchervish1985,West1992c,West1992d}
and generalizes other ensembling strategies as a special case. In
particular, \citet{mcalinn2019dynamic} developed a specific form
of dynamic BPS for time series data, utilizing a random walk DLM as
its synthesis function. Given its empirical success \citep{bianchi2018large,mcalinn2017multivariate,capek2020macroeconomic,mcalinn2017dynamic,aastveit2022quantifying,bernaciak2022loss},
and relation to the random walk DLM, a natural question is whether
the results in Sections~\ref{sec:minmax} and \ref{App:semimart}
also hold in this context.

\subsection{Problem set-up}

In the ensemble prediction set-up, a decision maker, $\mathcal{D}$,
is interested in forecasting a quantity $y\in\mathbb{R}$, and solicits
$J$ agents, where agents encompass models, forecasters, institutions,
etc. Agents are denoted as $\mathcal{A}_{j}$, $j\in J.$ Each agent,
$\mathcal{A}_{j}$, produces a predictive distribution, $h_{j}(\hat{y}_{j})$,
which comprises the set, $\mathcal{H}=\left\{ h_{1}(\cdot),\cdots,h_{J}(\cdot)\right\} $.

In the BPS framework, the set, $\mathcal{H}$, is synthesized via
Bayesian updating with the posterior of the form, 
\begin{equation}
p(y|\mathcal{H})=\int_{\hat{\boldsymbol{y}}}\alpha(y|\hat{\boldsymbol{y}})\prod_{j=\seq1J}h_{j}(\hat{y}_{j})d\hat{\boldsymbol{y}},\label{eq:theorem1}
\end{equation}
where $\hat{\boldsymbol{y}}=\hat{y}_{\seq1J}=(\hat{y}_{1},\ldots,\hat{y}_{J})^{\top}$ is
a $J-$dimensional latent vector and $\alpha(y|\hat{\boldsymbol{y}})$ is
a conditional probability density function for $y$ given $\hat{\boldsymbol{y}}$,
called the synthesis function. Eq.~\eqref{eq:theorem1} is only a
coherent Bayesian posterior if it satisfies the consistency condition\footnote{Specifically, the consistency condition states that given $\mathcal{D}$'s
prior, $p(y)$, and her prior expectation of what the agents will
produce before observing the forecasts, $m(\hat{\boldsymbol{y}})$, her
priors have to be consistent: $p(y)=\int_{\hat{\boldsymbol{y}}}\alpha(y|\hat{\boldsymbol{y}})m(\hat{\boldsymbol{y}})d\hat{\boldsymbol{y}}$,
for eq.~\eqref{eq:theorem1} to be a coherent Bayesian posterior.} \citep[see,][]{GenestSchervish1985,West1992c,West1992d,mcalinn2019dynamic}.

For time series data, \citet{mcalinn2019dynamic} developed a dynamic
specification of BPS. In the dynamic setting, $\mathcal{D}$ is sequentially
predicting a time series $y_{t},t=1,2,\ldots,$ and receives, for
all periods, forecast densities from $\mathcal{A}_{j}$, $j\in J.$
At each time $t,$ $\mathcal{D}$ aims to forecast $y_{t+1}$ and
receives the set $\mathcal{H}_{t+1}=\left\{ h_{1,{t+1}}(\hat{y}_{1,{t+1}}),\cdots,h_{J,{t+1}}(\hat{y}_{J,{t+1}})\right\} $,
a collection of agent predictive densities produced at $t$, from
the agents. The full information set used by $\mathcal{D}$ is thus
$\{\y_{\seq1{t}},\ \mH_{\seq1{t+1}}\}.$ As time passes, $\mathcal{D}$
learns about the characteristics of the agents (bias, dependencies, etc.).
Thus, the Bayesian model will involve parameters for which $\mathcal{D}$
updates information over time. From eq.~\eqref{eq:theorem1}, $\mathcal{D}$
has a time $t$ distribution for $y_{t+1}$ of the form 
\begin{equation}
p(y_{t+1}|\bPhi_{t+1},y_{\seq1{t}},\mH_{\seq1{t+1}})=\int\alpha_{t}(y_{t+1}|\hat{\boldsymbol{y}}_{t+1},\bPhi_{t+1})\prod_{j=\seq1J}h_{j,{t+1}}(\hat{y}_{j,{t+1}})d\hat{y}_{j,{t+1}},\label{eq:theorem}
\end{equation}
where $\bPhi_{t+1}=(\theta_{0,{t+1}},\btheta_{t+1},v_{t+1},\boldsymbol{W}_{t+1})$,
and the synthesis function, $\alpha_{t}(y_{t+1}|\hat{\boldsymbol{y}}_{t+1},\bPhi_{t+1})$,
in \citet{mcalinn2019dynamic} is specified as a standard dynamic
linear model eqs.~\ref{eq:DLMa}-\ref{eq:DLMc}.

Finally, the transition probability can be obtained by integrating
out the agent processes, $\hat{\boldsymbol{y}}_{t+1}$: 
\begin{align}
p_{t} & \left(y_{t+1}\left|\left\{ y_{s},\hat{\boldsymbol{y}}_{s}\right\} _{s=1}^{t},\theta_{y_{t+1}}^{*},\boldsymbol{\theta}_{t+1}^{*},v_{t+1}\right.\right)\\
 & =\int_{\mathbb{R}^{J}}\frac{1}{\sqrt{2\pi v_{t+1}}}\exp\left(-\frac{\left(y_{t+1}-\left\langle \boldsymbol{\theta}_{t+1}^{\prime},\hat{\boldsymbol{y}}^{\prime}_{t+1}\right\rangle \right)^{2}}{2v_{t+1}}\right)\sum_{j}h_{j,t+1}\left(\hat{y}_{j,t+1}\right)d\hat{y}_{j,t+1},\label{eq:TransitionFS}
\end{align}
where $\hat{\boldsymbol{y}}_{t}^{\prime}$ is
$\left[1,\hat{\boldsymbol{y}}_{t}\right]^{\top}$.

\subsection{Predictive minimaxity for predictive synthesis}

Given that the problem setting now involves densities to be synthesized,
we show that Theorem.~\ref{thm:Minimax} holds for the convolution
case: \begin{theorem} \textit{When Theorem.~\ref{thm:Minimax} holds,
i.e., when the random walk DLM is path-wise minimax, the convolution
of BPS by $h_{j,t+1}(\hat{y}_{j,t+1})$, for $j=1{:}J$, is also minimax.} 
\begin{proof}
From the proof of Theorem.~\ref{thm:Minimax}, the KL risk is risk
constant. Therefore, 
\[
\int_{\mathbb{R}^{J}}R_{\mathsf{KL}}\left(\left(\theta_{0,t+1}^{*},\boldsymbol{\theta}_{t+1}^{*},\sigma_{t+1}\right),\hat{p}_{t}^{\frac{1}{v_{0}}\mathsf{m}}\right)\prod_{j=1}^{J}h_{j}\left(d\hat{y}_{j,t+1}\right)=R_{\mathsf{KL}}\left(\left(\theta_{0,t+1}^{*},\boldsymbol{\theta}_{t+1}^{*},\sigma_{t+1}\right),\hat{p}_{t}^{\frac{1}{v_{0}}\mathsf{m}}\right).
\]
Then, from the log-sum inequality, we have 
\begin{align*}
R_{\mathsf{KL}}\left(\left(\theta_{0,t+1}^{*},\boldsymbol{\theta}_{t+1}^{*},\sigma_{t+1}\right),\int_{\mathbb{R}^{J}}\hat{p}_{t}^{\frac{1}{v_{0}}\mathsf{m}}\prod_{j=1}^{J}h_{j}\left(d\hat{y}_{j,t+1}\right)\right)\leqq\\
\int_{\mathbb{R}^{J}}R_{\mathsf{KL}}\left(\left(\theta_{0,t+1}^{*},\boldsymbol{\theta}_{t+1}^{*},\sigma_{t+1}\right),\hat{p}_{t}^{\frac{1}{v_{0}}\mathsf{m}}\right)\prod_{j=1}^{J}h_{j}\left(d\hat{y}_{j,t+1}\right).
\end{align*}
\end{proof}
\end{theorem}

The Gaussianity assumptions in the above theorem can also be relaxed
using semi-martingale processes, as in Section~\ref{App:semimart}.
This result proves that dynamic BPS, as proposed in \citet{mcalinn2019dynamic},
produces exact minimax predictive distributions.

Additionally, the above results show that linear combinations of forecasts
(e.g. equal weight averaging, Bayesian model
averaging, etc.) cannot produce predictive distributions that are
exact minimax. This is because $\max_{\theta_{0,t+1}^{*}}\mathsf{KL}\left(\left(\theta_{0,t+1}^{*},\boldsymbol{\theta}_{t+1}^{*}\right),\hat{p}_{t}\right)$
can be infinite due to the BPS intercept, $\theta_{0,t+1}^{*}$,
being under-parameterized for linear combination methods. This not
only shows that BPS is theoretically better (in terms of being minimax),
but that no linear combination of forecasts can be better under this criterion.

\section{Empirical Applications\label{sec:emst}}

We consider three datasets from epidemiology, climatology, and economics
(for a simulation study, see Appendix~\ref{sec:sim}). All three
datasets have common characteristics that highlight our theoretical
results: they are all non-stationary time series data, a ``true''
model is not something that can be expected to obtain, and future
predictions are tied to some decision making process, where uncertainty
quantification is critical.

\subsubsection*{Weekly average new cases of COVID-19 in Tokyo}

The first dataset concerns predicting the weekly average new cases
of COVID-19 in Tokyo, Japan. The prediction of positive cases of COVID-19
has become a critical task in tackling the pandemic, drawing considerable
interest \citep[see, e.g.,][]{awan2020prediction,gecili2021forecasting}.
This is because accurate predictions of new cases, as well as their
variation, have been shown to be imperative for resource management
in hospitals and local governments. For this application, we use daily
data on new cases to predict the next week's average new cases, as
the decision maker's interests are not necessarily in new cases ``tomorrow,''
but in the coming week, in order to, e.g., prepare medical resources.
The dataset, taken from the Tokyo Metropolitan Government, begins
on 2020/1/16, when the collection of data began, and ends at the end
of 2021. Our interest is to predict over the entirety of 2021, where
Tokyo experienced three major waves of new cases, including the Delta
variant, and began to experience the Omicron variant.

\subsubsection*{Monthly global mean sea level }

The second dataset aims to predict the monthly global mean sea level
from TOPEX and Jason Altimetry \citep{nerem2010estimating,masters2012comparison}.
The global mean sea level is a key indicator for climate change, reflecting
the ocean's thermal expansion, meltwater from mountain glaciers, and
discharge from the Greenland and Antarctic ice sheets. As such, the
prediction of the global mean sea level is essential in understanding
future trends in climate change. The satellite data is measured every
ten days, though the focus is on the monthly (30-day cycle) trends
of sea level, and its prediction. This is done, as with the COVID
dataset, because shorter frequency (i.e. 10-day cycle) is prone to
local fluctuations that may not be relevant for climate change, and
might not be of interest to predict. We consider the monthly frequency
to be of interest in this application in order to gauge the accuracy
of ensemble strategies.

\subsubsection*{Monthly aggregate inflation in the U.S.}

The third is a large-scale macroeconomic dataset to forecast monthly
U.S. inflation, a context of interest in the field \citep{Cogley2005,Nakajima2010}.
Specifically, forecasting inflation is one of the central bank's key
mandates, where that information is used to set policy. Typical horizon
of interest is 1-, 3-, 12-, and 24-months ahead. Although long term
forecasts can be, and are, cast as one step ahead forecasts within
the BPS framework, as BPS can directly calibrate to the horizon of
interest, we focus on the shortest horizon of interest, 1-month, as
this exemplifies the theoretical results the most. For this study,
we utilize a high-dimensional panel of $N=128$ monthly macroeconomic
and financial covariates from \citet{mccracken2016fred} (details
of each variable can be found therein). This context of using a large
dataset for economic forecasting is of topical interest, as leveraging
large datasets for decision making is crucial for central banks setting
their policy.

\subsection{Application design}

\subsubsection{Agent models}

All three applications are conducted in a similar manner, with some
variation to suit the context. We first divide the dataset equally
into the training and evaluation set. The first half of the training
set is used to build the individual models, in parallel. This is 2020/1/16--2020/6/27
for the COVID dataset, 1993:01--2002:05 for the sea level dataset,
and 1986:01--1993:06 for the economic dataset. Specifically, we use
a random walk DLM, for each agent model, $j=1{:}J$, 
\begin{subequations}
\label{eq:dlmde} 
\begin{align}
y_{t} & =\bbeta_{tj}^{\top}\x_{t-1,j}+\epsilon_{tj},\quad\epsilon_{tj}\sim N(0,\nu_{tj}),\label{eq:dlmde0}\\
\bbeta_{tj} & =\bbeta_{t-1,j}+\u_{tj},\quad\u_{tj}\sim N(0,\nu_{tj}\U_{tj}),\label{eq:dlmde1}
\end{align}
\end{subequations}
 where the coefficients follow a random walk and the observation variance
evolves with discount stochastic volatility. Priors for each predictive
regression are set as $\bbeta_{0j}|v_{0j}\sim N(\m_{0j},(v_{0j}/s_{0j})\I)$
with $\m_{0j}={\bf 0}'$ and $1/v_{0j}\sim G(n_{0j}/2,n_{0j}s_{0j}/2)$
with $n_{0j}=10,s_{0}=0.01$.

In the COVID dataset, each model is built with different autoregressive
lags, covariates, and day of the week indicators to capture the variation
in testing capabilities. Since the data are daily, and the interest
is the coming weekly average, we consider four models: AR(1)+Policy
(Policy: two indicators for the two levels of state of emergency),
AR(1)+Vax (Vax: six covariates for the vaccination rates for the first,
second, and third shots for the entire population and 65 and above),
AR(1)+Mob+Temp (Mob: six covariates for weekly average Google mobility
for retail and recreation, grocery and pharmacy, parks, stations,
workplaces, and residential; Temp: three covariates of weekly average
rainfall, temperature, and humidity), AR(3) and AR(7).

For the sea level dataset, each model is built with different autoregressive
lags: AR(1), AR(3), AR(9), and AR(18), as the data are monthly (30-day
cycle).

For the economic dataset, the covariates are the $N=128$ monthly
macroeconomic and financial covariates, partitioned into eight groups
based on the existing qualitative classification. The eight main categories
are: Output and Income, Labor Market, Consumption and Orders, Orders
and Inventories, Money and Credit, Interest Rate and Exchange Rates,
Prices, and Stock Market.

For the next half of the training data, while the individual DLMs
are sequentially updated in parallel, we also begin calibrating and
training the ensemble methods, including dynamic BPS. This is 2020/6/28--2020/12/31
for the COVID dataset, 2002:06--2010:12 for the sea level dataset,
and 1986:01-2000:12 for the economic dataset.

Finally, the latter half of the entire dataset is used as the evaluation
period, where we compare predictive performances of the ensemble methods
in an online manner, mirroring real world situations. This is 2021/1/1--2021/12/31
for the COVID dataset, 2011:01--2020:12 for the sea level dataset,
and 2001:01-2015:12 for the economic dataset. Updating of the individual
models, as well as competing strategies, is done sequentially for
each $t$ during this period.

We compare the mean squared forecast error (MSFE) and the log predictive
density ratio (LPDR). If we assume
a Gaussian process, the performance relation in KL risk is equivalent
to the performance relation in the MSE risk. The log predictive density
ratios (LPDR) for each $t$ is 
\begin{equation*}
\mathrm{LPDR}_{\seq1t}=\sum_{i=\seq1t}\mathrm{log}\{p_{*}(y_{t+1}|y_{1{:}t})/p_{\mathrm{BPS}}(y_{t+1}|y_{1{:}t})\}
\end{equation*}
where $p_{*}(y_{t+1}|y_{1{:}t})$ is the predictive density of the
model being compared with. This metric directly corresponds to the KL divergence. Comparing both the MSFE and LPDR provides
a more holistic assessment of the predictive performance.

\subsubsection{{Competing strategies}}

We compare our framework against each agent DLM and benchmark ensemble
strategies. Our goal is to show that dynamic BPS is, not only superior
to other ensemble methods, but superior to the forecasts that it is
ensembling; what would be considered a necessary requirement for ensemble
methods to achieve.

Specifically, we compare dynamic BPS against equal weight averaging,
BMA, Mallows $C_{p}$ averaging of \citet{hansen2007least}, and the
exponential weights algorithm of \citet{littlestone1994weighted}.
Equal weight averaging is standard in much of the forecast combination
literature and ensemble learning literature, being almost ubiquitous
in the latter. While the approach of taking the arithmetic mean of
forecasts is very simple, it is, nonetheless, considered a benchmark
that is hard to beat in real data applications \citep{Genre2013}.
On the other hand, BMA has the benefit of asymptotically converging
to the true model, when the true model is nested ($\mM$-closed),
though it tends to converge to the ``wrong" model in an $\mM$-open
setting. This is because, despite its name, BMA is more of a selection
strategy rather than an ensemble strategy. In almost all real applications,
including this one, we cannot assume that the true model is nested,
which guarantees that BMA will converge to the ``wrong" model. Nonetheless,
it is standard in much of the Bayesian literature. Mallows $C_{p}$
averaging has the property of being asymptotically optimal, achieving
the lowest possible squared error in a class of discrete model average
estimators. The exponential weights algorithm, otherwise called the
weighted majority algorithm, bounds the regret of not choosing the
best agent.

The prior specification for dynamic BPS is $\theta_{0,0}|v_{0}\sim N(0,v_{0}/s_{0})$,
$\btheta_{0}|v_{0}\sim N(\m_{0},(v_{0}/s_{0})\sigma^{2})$ with $\m_{0}=(0,\one'/J)'$
and $1/v_{0}\sim G(n_{0}/2,n_{0}s_{0}/2)$ with $n_{0}=10,s_{0}=0.002$
(the variance is scaled by 1/100 for the sea level dataset to reflect
the variation in the data). The discount factors are set to $(\beta,\delta)=(0.95,0.99)$.

\subsection{Results}

\begin{table}[t!]
\centering \caption{Out-of-sample forecast performance}
\vspace{0.1in}
 \begin{justify} {\footnotesize{}{}{}{}{This table reports the
out-of-sample comparison of dynamic BPS against each individual model,
equal weight averaging (EW), and BMA for forecasting the three real-world
datasets. Performance comparison is based on the ratio of Mean Squared
Forecast Error (MSFE), MSFE$_{*}$/MSFE$_{BPS}$, where ${*}$ denotes
the method compared against, and the Log Predictive Density Ratio
(LPDR).}} \end{justify} 
\global\long\def\arraystretch{1.2}%
\vspace{0.5em}
 \textbf{Panel A: }Forecasting Weekly Average New COVID Cases in 2021
\vspace{0.2cm}
 \\
 \resizebox{1\textwidth}{!}{%
\begin{tabular}{lcccccccccc}
\hline 
\multicolumn{6}{c}{Agent models} & EW  & BMA  & Mallows  & ExW  & BPS \tabularnewline
\hline 
 & AR(1)+Policy  & AR(1)+Vax  & AR(1)+Mob+Temp  & AR(3)  & AR(7)  &  &  &  &  & \tabularnewline
\hline 
MSFE  & 1.9952  & 1.3195  & 1.9193  & 1.7313  & 1.5722  & 1.4980  & 1.3920  & 1.5574  & 1.6696  & 1.0000 \tabularnewline
LPDR  & -100.17  & -38.60  & -98.07  & -81.40  & -76.42  & -66.73  & -54.57  & -74.34  & -69.00  & -- \tabularnewline
\hline 
\end{tabular}}

\vspace{1em}
 \textbf{Panel B: }Forecasting Monthly Global Mean Sea Level for 2011:01--2020:12
\vspace{0.2cm}
 \\
 \resizebox{0.8\textwidth}{!}{%
\begin{tabular}{lccccccccc}
\hline 
\multicolumn{5}{c}{Agent models} & EW  & BMA  & Mallows  & ExW  & BPS \tabularnewline
\hline 
 & AR(1)  & AR(3)  & AR(9)  & AR(18)  &  &  &  &  & \tabularnewline
\hline 
MSFE  & 1.4421  & 1.0508  & 1.0738  & 1.2013  & 1.0855  & 1.0726  & 1.0461  & 1.4149  & 1.0000 \tabularnewline
LPDR  & -25.17  & -5.26  & -4.89  & -10.62  & -9.94  & -4.87  & -6.77  & -24.16  & -- \tabularnewline
\hline 
\end{tabular}}

\vspace{1em}
 \textbf{Panel C: }Forecasting Monthly Inflation for 2001:01-2015:12
\vspace{0.2cm}
 \\
 \resizebox{1\textwidth}{!}{%
\begin{tabular}{lccccccccccccc}
\hline 
\multicolumn{9}{c}{Group-Specific Models} & EW  & BMA  & Mallows  & ExW  & BPS \tabularnewline
\hline 
 & %
\begin{tabular}{@{}c@{}}
Output \&\tabularnewline
Income\tabularnewline
\end{tabular} & %
\begin{tabular}{@{}c@{}}
Labor\tabularnewline
Market\tabularnewline
\end{tabular} & %
\begin{tabular}{@{}c@{}}
Consump.\tabularnewline
\end{tabular} & %
\begin{tabular}{@{}c@{}}
Orders \&\tabularnewline
Invent.\tabularnewline
\end{tabular} & %
\begin{tabular}{@{}c@{}}
Money\tabularnewline
\& Credit\tabularnewline
\end{tabular} & %
\begin{tabular}{@{}c@{}}
Int. Rate \& \tabularnewline
Ex. Rates\tabularnewline
\end{tabular} & Prices  & %
\begin{tabular}{@{}c@{}}
Stock\tabularnewline
Market\tabularnewline
\end{tabular} &  &  &  &  & \tabularnewline
\hline 
MSFE  & 1.4715  & 1.2003  & 12.8039  & 1.7601  & 1.6368  & 4.3103  & 1.1748  & 6.0074  & 2.0618  & 1.7601  & 1.2577  & 1.2779  & 1.0000 \tabularnewline
LPDR  & -40.48  & -42.05  & -233.09  & -59.15  & -56.34  & -134.18  & -20.00  & -171.21  & -88.81  & -60.40  & -21.91  & -28.22  & -- \tabularnewline
\hline 
\end{tabular}} \label{tab:Table003} 
\end{table}

Panels A, B, and C of Table~\ref{tab:Table003} show the predictive
results for the COVID dataset, sea level dataset, and macroeconomic
dataset, respectively. Over the three datasets, dynamic BPS improves
the out-of-sample forecasting accuracy compared to everything we consider,
including all individual models and the four ensemble methods. This
result is consistent for all datasets, though the performance gains
for BPS differ amongst datasets. Most notably, the sea level datasets
display lower gains compared to the other two, although the gains
are still significant. This is likely due to the consistent trend
exhibited in the data, where large fluctuations are not as clear as
the other two. Nonetheless, dynamic BPS shows improved performance
for both point and density evaluations.

In particular, for the COVID dataset and the macroeconomic dataset,
dynamic BPS consistently improves over the other ensemble methods
by large margins for both point and density forecasts. For the COVID
data, specifically, dynamic BPS has an LPDR of approximately -40 to
-100 across all ensemble methods compared against. The difficulty
of standard ensemble methods in dealing with non-stationary data is
also clear, with some even underperforming the agent models. For RMSE,
Mallows $C_{p}$ edges out only in the sea level dataset, and for
LPDR, Mallows $C_{p}$ slightly outperforms the best agent model in
the COVID dataset and BMA is practically equivalent to the best model
in the sea level dataset. Notably, the best performing agent model
in terms of RMSE is not the same as the LPDR, except for the macroeconomic
dataset. The standard ensemble methods seemingly struggle with this
dissonance in performance, though dynamic BPS is able to successfully
synthesize the information to outperform all in both metrics. This
shows that, while the theory we presented is with regard to KL risk,
dynamic BPS can outperform other ensemble methods for different decision
making problems.

Finally, we look at the interpretability of dynamic BPS through its
on-line coefficients. Figure.~\ref{fig:covid-coeff} shows the on-line
mean coefficients of BPS for the weekly COVID prediction dataset.
Most notably, the intercept is quite volatile compared to the other
coefficients. As the intercept can be viewed as the model set uncertainty
(uncertainty not captured by the agent models), this shows how the
different lags and covariates are not sufficient to predict weekly
COVID cases. The intercept also plays a critical role in the theoretical
results, and the fact that the intercept adapts over time at this
magnitude, as seen in the figure, shows this in action.

Looking at the other coefficients, several points stand out. First,
as an overall trend, the model with vaccination rates increases in
importance over time. This suggests the importance of vaccination
information to predict COVID cases, with its importance increasing
as the vaccination rate increases. Second, the AR(7) model, which
captures longer trends, spikes during increased COVID cases. This
could be explained by the lagged nature of COVID diagnosis. Notably,
this model is inversely correlated with the AR(1)+Mob+Temp model,
which models mobility and weather data. Third, the model with policy
indicators is effective during the first semi-emergency measure, though
its importance diminishes throughout the latter emergencies. This
is in line with the idea that these measures are not as effective
in the latter iterations, due to people adapting or being complacent
towards these measures.

As these results show, the on-line coefficients provide critical insight
into the predictive process, as well as highlight the theoretical
findings of this paper. In particular, the trend exhibited in the
intercept supports the theoretical results that the intercept plays
a crucial role. Added to the fact that dynamic BPS improves over other
ensemble methods, the result empirically supports the theoretical
findings.

\begin{figure}[t!]
\caption{Weekly COVID-19 forecasting: Dynamic BPS coefficients}
\centering \vspace{0.1in}
 \begin{justify} {\footnotesize{}{}{}{}{On-line posterior means
of BPS model coefficients sequentially computed each day over 2021.
The bottom bar plot displays weekly COVID-19 cases. Light gray background
indicates dates during semi-emergency measures and dark gray background
indicates dates during states of emergency.}} \end{justify} \includegraphics[width=1\textwidth]{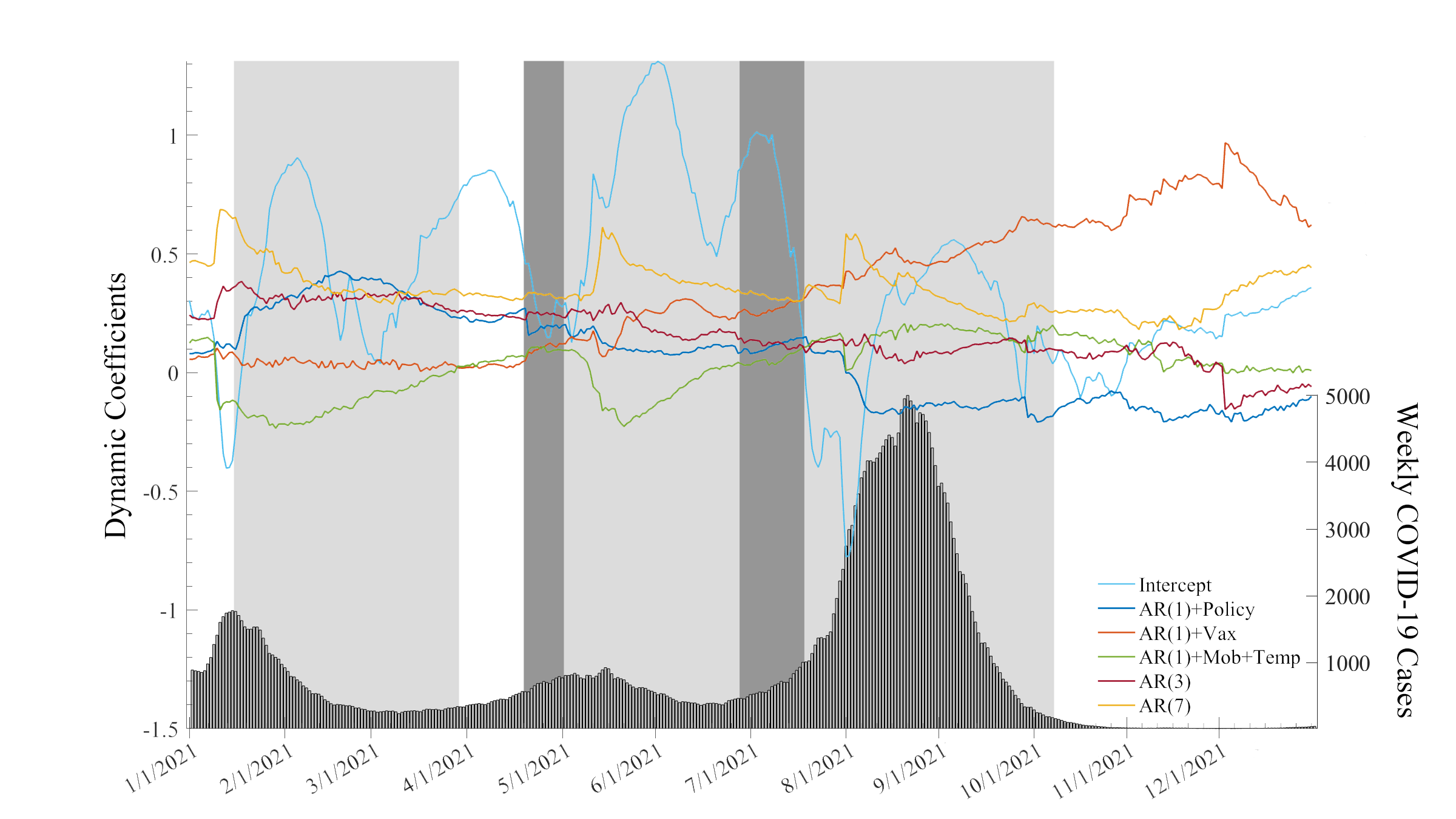}
\label{fig:covid-coeff} 
\end{figure}

\section{Summary and Additional Comments}

Despite the importance of non-stationary time series forecasts in
many domains, theoretical results are virtually non-existent. This
is due to the dynamic nature of the data, where assumptions-- such
as i.i.d.--, used in the existing literature to show asymptotic results,
do not hold. This has stymied theoretical evaluations of predictive
performances, especially when all models are misspecified. To evaluate
the performance of statistical methods in this setting, we define
the Kullback-Leibler risk for non-stationary data. Using this criterion,
we prove that a random walk DLM is exact minimax, providing finite
sample theoretical support for the method over other existing methods.
We then show that, extending the problem to predictive density synthesis,
dynamic BPS is also exact minimax. Through three relevant applications
from epidemiology, climatology, and economics, we show that BPS; (i)
outperforms individual models that it synthesizes, (ii) outperforms
other ensemble methods, as the theory indicates, and (iii) provides
useful information for decision making. All three real data applications
reinforce the validity and applicability of the theoretical result.

We believe the results presented in this paper open up several avenues
of research related to time series forecasting and decision making,
as well as causal inference with the potential outcomes framework.
If, for example, the interest is in predicting potential outcomes
for time series data, this framework can provide some theoretical
guarantee, even if the data cannot be assumed to be i.i.d. Different
problem settings will require further development of the framework,
though having the Kullback-Leibler risk defined is the first step
towards solving each problem.

One notable direction for future development is in the direction of
multivariate models, which is also relevant for policy decision making.
However, extending the proposed framework to a multivariate setting
is not trivial, due to the shift transform only holding for univariate
or very specific multivariate data. For this purpose, a different
approach to replace the shift transform is necessary, which will be
future work. Further connections can be made to the generalized Bayes
framework \citep{bissiri2016general}, as seen in \citet{bernaciak2022loss,tallman2022bayesian},
where the generalized Bayes framework is formulated as a special case
of BPS (i.e. the synthesis function is a loss function). The minimax
result, which is decision theoretic, could potentially produce fruitful
results in this direction as well.

\bibliographystyle{chicago}
\bibliography{reference}

\newpage{}

\appendix

\section*{Supplementary material for \protect \protect \protect \\
 Equivariant online predictions of non-stationary time series}

\section{Proof of Lemma \ref{lemma:Invariant}\label{sec:lemMinmax}}

The posterior distribution, $\pi\left(\boldsymbol{\theta}_{t}^{\prime}\left|\left\{ y_{s},\boldsymbol{x}_{s}\right\} _{s=1}^{t}\right.\right)$,
for $\boldsymbol{\theta}_{t}^{\prime}$ is updated as 
\[
\pi\left(\boldsymbol{\theta}_{t}^{\prime}\left|\left\{ y_{s},\boldsymbol{x}_{s}\right\} _{s=1}^{t}\right.\right)=\frac{q_{t}\left(y_{t}\left|\boldsymbol{\theta}_{t}^{\prime},\boldsymbol{x}_{t}^{\prime}\right.\right)\pi\left(\boldsymbol{\theta}_{t}^{\prime}\left|\mathcal{F}_{t-1}\right.\right)}{\int q_{t}\left(y_{t}\left|\boldsymbol{\theta}_{t}^{\prime},\boldsymbol{x}_{t}^{\prime}\right.\right)\pi\left(\boldsymbol{\theta}_{t}^{\prime}\left|\mathcal{F}_{t-1}\right.\right)d\boldsymbol{\theta}_{t}^{\prime}},
\]
via the Bayes theorem, and the predictive distribution for $\boldsymbol{\theta}_{t}^{\prime}$
(the prior distribution, $\pi\left(\boldsymbol{\theta}_{t}^{\prime}\left|\mathcal{F}_{t-1}\right.\right)$,
at $t$) is given as 
\begin{alignat*}{1}
\pi\left(\boldsymbol{\theta}_{t}^{\prime}\left|\mathcal{F}_{t-1}\right.\right) & =\int\pi\left(\boldsymbol{\theta}_{t}^{\prime}\left|\boldsymbol{\theta}_{t-1}^{\prime}\right.\right)\pi\left(\boldsymbol{\theta}_{t-1}^{\prime}\left|\mathcal{F}_{t-1}\right.\right)d\boldsymbol{\theta}_{t-1}^{\prime}.
\end{alignat*}
Since we assume that $\boldsymbol{\theta}_{t}^{\prime}$ follows a
random walk process (eq.~\ref{eq:RWstate}), the transition probability
is $\pi\left(\boldsymbol{\theta}_{t}^{\prime}\left|\boldsymbol{\theta}_{t-1}^{\prime}\right.\right)=N\left(\boldsymbol{\theta}_{t-1}^{\prime},W_{t}\right)$
and $\boldsymbol{\theta}_{t}^{\prime}$ is invariant with regard to
the scale-shift transformation, $\boldsymbol{\theta}_{t}^{\prime}\rightarrow\widetilde{\boldsymbol{\theta}^{\prime}}_{t}$:
$\pi\left(\boldsymbol{\theta}_{t}^{\prime}\left|\boldsymbol{\theta}_{t-1}^{\prime}\right.\right)=\pi\left(\widetilde{\boldsymbol{\theta}^{\prime}}_{t}\left|\widetilde{\boldsymbol{\theta}^{\prime}}_{t-1}\right.\right)$.
Further, since we assume that the prior distribution of $\boldsymbol{\theta}_{t}^{\prime}$,
at $t=0$, is a Lebesgue measure, $\mathsf{m}\left(\cdot\right)$,
from the scale-shift invariance of Lebesgue measures, 
\[
\mathsf{m}\left(\theta_{0,0}\right)=\mathsf{m}\left(\widetilde{\theta}_{0,0}\right),\ \mathsf{m}\left(\boldsymbol{\theta}_{0}\right)=\mathsf{m}\left(\widetilde{\boldsymbol{\theta}}_{0}\right),
\]
and iterating through the Bayes theorem, we have 
\begin{alignat*}{1}
\pi\left(\boldsymbol{\theta}_{t}^{\prime}\left|\left\{ y_{s},\boldsymbol{x}_{s}\right\} _{s=1}^{t}\right.\right) & =\pi\left(\widetilde{\boldsymbol{\theta}^{\prime}}_{t}\left|\left\{ y_{s},\boldsymbol{x}_{s}\right\} _{s=1}^{t}\right.\right).
\end{alignat*}
Therefore, 
\[
q_{t}^{\mathsf{m}}\left(\widetilde{y}_{t+1}\left|y_{t}\right.\right)=q_{t}^{\mathsf{m}}\left(y_{t+1}\left|y_{t}\right.\right),
\]
and thus the KL risk is constant for all $\left(\mathsf{a}^{*},\boldsymbol{\theta}_{t}^{*}\right)$:
\[
R_{\mathsf{KL}}\left(\left(\mathsf{a}^{*},\boldsymbol{\theta}_{t}^{*}\right),q_{t}^{\mathsf{m}}\right)=R_{\mathsf{KL}}\left(\left(0,0\right),q_{t}^{\mathsf{m}}\right)=c.
\]

\section{Proof of Theorem \ref{thm:Minimax} \label{sec:Proof-of-Minimax}}

\subsection{Risk constant \label{subsec:risk const}}

First, we show that the predictive distribution of the random walk
DLM is risk constant. This is done, as mentioned previously, because
we are only concerned with the minimaxity of decision making at time
$t$, for the predictive value, $y_{t+1}$. For decision making at
time $t+1$, we simply add, $-\left(\frac{1}{\sigma_{t+1}}\theta_{0,t+1}^{*}+\frac{1}{\sigma_{t+1}}\left\langle \boldsymbol{\theta}_{t+2}^{*},\boldsymbol{x}_{t+2}\right\rangle \right)$,
and multiply $\frac{1}{\sigma_{t+1}}$ to all past time points. Here,
the change in the transition, $p_{t}\left(\cdot\left|\left\{ y_{s},\boldsymbol{x}_{s}\right\} _{s=1}^{t},\theta_{0,t+1}^{*},\boldsymbol{\theta}_{t+1}^{*},\sigma_{t+1},\boldsymbol{x}_{t+1}\right.\right)$,
at $t$ is expressed as 
\[
p_{t}\left(y_{t+1}\left|\left\{ y_{s},\boldsymbol{x}_{s}\right\} _{s=1}^{t},\theta_{0,t+1}^{*},\boldsymbol{\theta}_{t+1}^{*},\sigma_{t+1},\boldsymbol{x}_{t+1}\right.\right)=p_{t}\left(\tilde{y}_{t+1}\left|\left\{ y_{s},\boldsymbol{x}_{s}\right\} _{s=1}^{t},0,0,1,\boldsymbol{x}_{t+1}\right.\right),
\]
and does not depend on the parameters, $\left(\theta_{0,t+1}^{*},\boldsymbol{\theta}_{t+1}^{*},\sigma_{t+1}\right)$.
Therefore, the KL risk is constant for any set of true parameters,
$\left(\theta_{0,t+1}^{*},\boldsymbol{\theta}_{t+1}^{*},\sigma_{t+1}\right)$:
\begin{alignat*}{1}
R_{\mathsf{KL}}\left(\left(\theta_{0,t+1}^{*},\boldsymbol{\theta}_{t+1}^{*},\sigma_{t+1}\right),\hat{p}_{t}^{\mathsf{m}}\right) & =R_{\mathsf{KL}}\left(\boldsymbol{\theta}_{t+1}^{\prime},v_{t+1},\hat{p}_{t}^{\mathsf{m}}\right)\\
 & =R_{\mathsf{KL}}\left(\widetilde{\boldsymbol{\theta}^{\prime}}_{t+1},\tilde{v}_{t+1},\hat{p}_{t}^{\mathsf{m}}\right)=c.
\end{alignat*}

The KL risk is constant between the transition probability of the
DGP, $$p_{t}\left(y_{t+1}\left|\left\{ y_{s},\boldsymbol{x}_{s}\right\} _{s=1}^{t},\theta_{0,t+1}^{*},\boldsymbol{\theta}_{t+1}^{*},\sigma_{t+1},\boldsymbol{x}_{t+1}\right.\right),$$
and the predictive distribution, $$\hat{p}_{t}^{\frac{1}{v_{0}}\mathsf{m}}\left(y_{t+1}\left|\left\{ y_{s},\boldsymbol{x}_{s}\right\} _{s=1}^{t},\boldsymbol{x}_{t+1}\right.\right).$$

\subsection{Extended Bayes \label{subsec:extend}}

We now consider an extended Bayes strategy given an initial prior,
$\rho_{k}\left(\boldsymbol{\theta}_{0}^{\prime},v_{0}\right)=N\left(0,\sigma_{k}^{2}I\right)\frac{1}{\left|v_{0}\right|^{1+c_{k}}}$,
and show that the Bayes risk limit is 
\[
B\left(\rho_{k},\hat{p}_{t}^{\rho_{k}}\right)\rightarrow B\left(\rho_{k},\hat{p}_{t}^{\frac{1}{v_{0}}\mathsf{m}}\right),\ \left(k\rightarrow\infty\right),
\]
when $\sigma_{k}^{2}\rightarrow\infty$ and $c_{k}\left(>0\right)\rightarrow0$
as $k\rightarrow\infty$.

The difference in KL risk of the above is 
\begin{align*}
 & R_{\mathsf{KL}}\left(\left(\theta_{0,t+1}^{*},\boldsymbol{\theta}_{t+1}^{*},\sigma_{t+1}\right),\hat{p}_{t}^{\frac{1}{v_{0}}\mathsf{m}}\right)-R_{\mathsf{KL}}\left(\left(\theta_{0,t+1}^{*},\boldsymbol{\theta}_{t+1}^{*},\sigma_{t+1}\right),\hat{p}_{t}^{\rho_{n}}\right)\\
= & \int_{Y}\mathbb{E}\left[\log\frac{\hat{p}_{t}^{\frac{1}{v_{0}}\mathsf{m}}\left(y_{t+1}\left|\left\{ y_{s},\boldsymbol{x}_{s}\right\} _{s=1}^{t},\boldsymbol{x}_{t+1}\right.\right)}{\hat{p}_{t}^{\rho_{k}}\left(y_{t+1}\left|\left\{ y_{s},\boldsymbol{x}_{s}\right\} _{s=1}^{t},\boldsymbol{x}_{t+1}\right.\right)}\right]d\mu_{Y}\left(y_{0},\cdots,y_{t}\right)\\
= & \int_{Y}\mathbb{E}\left[\log\frac{\int\int\int\hat{p}_{t}\left(y_{t+1}\left|\boldsymbol{\theta}_{t+1}^{\prime},\boldsymbol{x}_{t+1}\right.\right)\pi^{\frac{1}{v_{0}}\mathsf{m}}\left(\boldsymbol{\theta}_{t+1}^{\prime},v_{t+1}\left|\left\{ y_{s},\boldsymbol{x}_{s}\right\} _{s=1}^{t}\right.\right)d\boldsymbol{\theta}_{t+1}^{\prime}dv_{t+1}}{\int\int\int\hat{p}_{t}\left(y_{t+1}\left|\boldsymbol{\theta}_{t+1}^{\prime},\boldsymbol{x}_{t+1}\right.\right)\pi^{\rho_{k}}\left(\boldsymbol{\theta}_{t+1}^{\prime},v_{t+1}\left|\left\{ y_{s},\boldsymbol{x}_{s}\right\} _{s=1}^{t}\right.\right)d\boldsymbol{\theta}_{t+1}^{\prime}dv_{t+1}}\right]d\mu_{Y}\left(y_{0},\cdots,y_{t}\right).
\end{align*}
The difference between, $\hat{p}_{t}^{\frac{1}{v_{0}}\mathsf{m}}\left(y_{t+1}\left|\left\{ y_{s},\boldsymbol{x}_{s}\right\} _{s=1}^{t},\boldsymbol{x}_{t+1}\right.\right)$
and $\hat{p}_{t}^{\rho_{k}}\left(y_{t+1}\left|\left\{ y_{s},\boldsymbol{x}_{s}\right\} _{s=1}^{t},\boldsymbol{x}_{t+1}\right.\right)$,
is the posterior parameter, $\pi\left(\boldsymbol{\theta}_{t}^{\prime},v_{t}\left|\left\{ y_{s},\boldsymbol{x}_{s}\right\} _{s=1}^{t}\right.\right)$.
The posterior, $\pi\left(\boldsymbol{\theta}_{t}^{\prime},v_{t}\left|\left\{ y_{s},\boldsymbol{x}_{s}\right\} _{s=1}^{t}\right.\right)$,
is given as the direct product of the posterior, $\boldsymbol{\theta}_{t}^{\prime},v_{t}$,
and the posterior, $v_{t}$, is a probability distribution that follows
the random variable, $\left(\prod_{i=1}^{t}\frac{\beta}{\gamma_{i}}\right)\frac{1}{v_{0}}$.

The two conditional probabilities, $\pi^{\frac{1}{v_{0}}\mathsf{m}}\left(\boldsymbol{\theta}_{t}^{\prime},v_{t}\left|\left\{ y_{s},\boldsymbol{x}_{s}\right\} _{s=1}^{t}\right.\right)$
and $\pi^{\rho_{k}}\left(\boldsymbol{\theta}_{t}^{\prime}\left|\left\{ y_{s},\boldsymbol{x}_{s}\right\} _{s=1}^{t}\right.\right)$,
can be obtained via the generalized Bayes formula for filtering \citep{Liptser-Shiryaev_01}.
From this, the posterior distribution is, 
\[
\pi^{\rho}\left(\boldsymbol{\theta}_{t_{n}}^{\prime}\left|\left\{ y_{s},\boldsymbol{x}_{s}\right\} _{s=1}^{t},\left\{ v_{s}\right\} _{s=0}^{t}\right.\right)=\int_{{Y}}\frac{\int_{\left(\mathbb{R}^{J}\right)^{t-1}}\beta_{t}\left(\boldsymbol{\theta}^{\prime},\boldsymbol{v},y\right)d\mu_{\theta}^{\rho}\left(\boldsymbol{\theta}_{0}^{\prime},\cdots,\boldsymbol{\theta}_{t-1}^{\prime}\right)}{\int_{\left(\mathbb{R}^{J}\right)^{t}}\beta_{t}\left(\boldsymbol{\theta}^{\prime},\boldsymbol{v},y\right)d\mu_{\theta}^{\rho}\left(\boldsymbol{\theta}_{0}^{\prime},\cdots,\boldsymbol{\theta}_{t}^{\prime}\right)}d\mu_{{Y}}\left({y}_{1},\cdots,{y}_{t}\right).
\]
Here, $\beta_{n}\left(\boldsymbol{\theta}^{\prime},\boldsymbol{v},y\right)$
is defined as 
\begin{alignat*}{1}
\beta_{n}\left(\boldsymbol{\theta}^{\prime},\boldsymbol{v},y\right) & =\exp\left\{ \sum_{k=1}^{t}\frac{\left\langle \boldsymbol{\theta}_{k}^{\prime},\boldsymbol{x}_{k}^{\prime}\right\rangle y_{k}}{v_{k}}-\frac{1}{2}\sum_{k=1}^{n}\frac{\left(\left\langle \boldsymbol{\theta}_{k}^{\prime},\boldsymbol{x}_{k}^{\prime}\right\rangle \right)^{2}}{v_{k}^{2}}\right\} \\
 & =\prod_{k=1}^{n}\exp\left\{ \frac{\left\langle \boldsymbol{\theta}_{k}^{\prime},\boldsymbol{x}_{k}^{\prime}\right\rangle y_{k}}{v_{k}}-\frac{1}{2}\frac{\left\langle \boldsymbol{\theta}_{k}^{\prime},\boldsymbol{x}_{k}\right\rangle ^{2}}{v_{k}^{2}}\right\} ,
\end{alignat*}
where each $\mu_{\theta}^{\rho},\mu_{{Y}}$ is a cylindrical measure
of processes, $\left\{ \boldsymbol{\theta}_{t}^{\prime}\right\} $
and $\left\{ \boldsymbol{x}_{t}\right\} $, at each $t=t_{k}$, and
the superscript, $\rho$, in $\mu_{\theta}^{\rho}$ specifies the
distribution of $\left(\boldsymbol{\theta}_{0}^{\prime}\right)$ as
$\rho\left(\cdot\right)$.

To evaluate the difference in KL risk, we define the predictive distribution,
\[
\hat{p}_{t}^{\rho}\left(y_{t+1}\left|\left\{ y_{s},\boldsymbol{x}_{s}\right\} _{s=1}^{t},\boldsymbol{x}_{t+1},\left\{ v_{s}\right\} _{s=0}^{t+1}\right.\right),
\]
for each process, $\left\{ \boldsymbol{x}_{t},v_{t}\right\} $, i.e.,
the path-wise predictive distribution, as 
\begin{align*}
 & \hat{p}_{t}^{\rho}\left(y_{t+1}\left|\left\{ y_{s},\boldsymbol{x}_{s}\right\} _{s=1}^{t},\boldsymbol{x}_{t+1},\left\{ v_{s}\right\} _{s=0}^{t+1}\right.\right)\\
= & \int_{\theta_{t+1}}\int_{\theta_{t}}\hat{p}_{t}\left(y_{t+1}\left|\boldsymbol{\theta}_{t+1}^{\prime},v_{t+1},\boldsymbol{x}_{t+1}^{\prime}\right.\right)\pi\left(\boldsymbol{\theta}_{t+1}^{\prime}\left|\boldsymbol{\theta}_{t}^{\prime}\right.\right)\pi^{\rho}\left(\boldsymbol{\theta}_{t_{n}}^{\prime}\left|\left\{ y_{s},\boldsymbol{x}_{s}\right\} _{s=1}^{t},\left\{ v_{s}\right\} _{s=0}^{t}\right.\right)d\boldsymbol{\theta}_{t}^{\prime}d\boldsymbol{\theta}_{t+1}^{\prime}.
\end{align*}
Here, $\pi^{\rho}\left(\boldsymbol{\theta}_{t_{n}}^{\prime}\left|\left\{ y_{s},\boldsymbol{x}_{s}\right\} _{s=1}^{t},\left\{ v_{s}\right\} _{s=0}^{t}\right.\right)$
is 
\[
\pi^{\rho}\left(\boldsymbol{\theta}_{t_{n}}^{\prime}\left|\left\{ y_{s},\boldsymbol{x}_{s}\right\} _{s=1}^{t},\left\{ v_{s}\right\} _{s=0}^{t}\right.\right)=\frac{\int_{\left(\mathbb{R}^{J}\right)^{n-1}}\beta_{n}\left(\boldsymbol{\theta}^{\prime},\boldsymbol{v},y\right)d\mu_{\theta}^{\rho}\left(\left(\boldsymbol{\theta}_{0}^{\prime}\right),\cdots,\left(\boldsymbol{\theta}_{t_{n-1}}^{\prime}\right)\right)}{\int_{\left(\mathbb{R}^{J}\right)^{n}}\beta_{n}\left(\boldsymbol{\theta}^{\prime},\boldsymbol{v},y\right)d\mu_{\theta}^{\rho}\left(\left(\boldsymbol{\theta}_{0}^{\prime}\right),\cdots,\left(\boldsymbol{\theta}_{t_{n}}^{\prime}\right)\right)},
\]
which is the path-wise posterior distribution for each process, $\left\{ \boldsymbol{x}_{t}\right\} $.
Then, $\hat{p}_{t}^{\rho}\left(y_{t+1}\left|\left\{ y_{s},\boldsymbol{x}_{s}\right\} _{s=1}^{t},\boldsymbol{x}_{t+1}\right.\right)$
can be expressed as the following. 
\begin{prop}
\label{prop:2} Consider the initial prior, $\rho\left(\theta_{0}\right)=\mathcal{N}\left(0,W_{0}\right)$.
The path-wise predictive distribution (for each process, $\left\{ \boldsymbol{x}_{t}\right\} $),
$\hat{p}_{t}^{\rho}\left(y_{t+1}\left|\left\{ y_{s},\boldsymbol{x}_{s}\right\} _{s=1}^{t},\boldsymbol{x}_{t+1}\right.\right)$,
can be expressed as, 
\begin{align*}
\hat{p}_{t}^{\rho}\left(y_{t+1}\left|\left\{ y_{s},\boldsymbol{x}_{s}\right\} _{s=1}^{t},\boldsymbol{x}_{t+1}\right.\right) & =\hat{p}_{t}^{\frac{1}{v_{0}}\mathsf{m}}\left(y_{t+1}\left|\left\{ y_{s},\boldsymbol{x}_{s}\right\} _{s=1}^{t},\boldsymbol{x}_{t+1}\right.\right)\\
 & \times\frac{\exp\left\{ -\left(\overline{\theta}_{1}\right)^{\top}\left(\tilde{\varOmega}_{0}\right)^{-1}\left(\overline{\theta}_{1}\right)\right\} \sqrt{\left(2\pi\left|\tilde{\varSigma}_{0}\right|\right)^{d}}\sqrt{\left(2\pi\left|\check{\varOmega}_{1}\right|\right)^{J}}}{\exp\left\{ -\left(\grave{\theta}_{1}\right)^{\top}\left(\check{\varOmega}_{0}\right)^{-1}\left(\grave{\theta}_{1}\right)\right\} \sqrt{\left(2\pi\left|\check{\varSigma}_{0}\right|\right)^{d}}\sqrt{\left(2\pi\left|\tilde{\varOmega}_{1}\right|\right)^{J}}},
\end{align*}
where $\hat{p}_{t}^{\frac{1}{v_{0}}\mathsf{m}}\left(y_{t+1}\left|\left\{ y_{s},\boldsymbol{x}_{s}\right\} _{s=1}^{t},\boldsymbol{x}_{t+1}\right.\right)$
is the predictive distribution when the initial prior is a Lebesgue
measure, and 
\begin{align*}
\tilde{\varSigma}_{0} & =\left(\tilde{\varOmega}_{1}^{-1}+W_{0}^{-1}\right)^{-1}, & \check{\varSigma}_{0} & =\left(\check{\varOmega}_{1}^{-1}+W_{0}^{-1}\right)^{-1}, & \overline{\theta}_{1} & =\left(\boldsymbol{x}_{1}^{\prime}V_{1}^{-1}\boldsymbol{x}_{1}^{\prime\top}+\tilde{\varOmega}_{2}^{-1}\right)^{-1}\left(\boldsymbol{x}_{1}^{\prime}V_{1}^{-1}\boldsymbol{x}_{1}^{\prime\top}\hat{\theta}_{1}+\tilde{\varOmega}_{2}^{-1}\overline{\theta}_{2}\right),\\
\tilde{\varOmega}_{0} & =\tilde{\varOmega}_{1}+W_{0}, & \check{\varOmega}_{0} & =\check{\varOmega}_{1}+W_{0}, & \grave{\theta}_{1} & =\left(\boldsymbol{x}_{1}^{\prime}V_{1}^{-1}\boldsymbol{x}_{1}^{\top}+\check{\varOmega}_{2}^{-1}\right)^{-1}\left(\boldsymbol{x}_{1}^{\prime}V_{1}^{-1}\boldsymbol{x}_{1}^{\prime\top}\hat{\theta}_{1}+\check{\varOmega}_{2}^{-1}\hat{\theta}_{2}\right).
\end{align*}
See Appendix~\ref{App:prop2detail} for details of the iterations. 
\end{prop}

\begin{proof}
See Appendix~\ref{App:prop2}. 
\end{proof}
From this, we immediately have the following corollary. 
\begin{cor}
\label{cor:2} Consider the initial prior, $\rho_{k}\left(\boldsymbol{\theta}_{0}^{\prime}\right)=N\left(0,\sigma_{k}^{2}I\right)$,
where the variance parameter, $\sigma_{k}$, is $\sigma_{k}\rightarrow\infty$
as $k\rightarrow\infty$. For each process $\left\{ \boldsymbol{x}_{t}\right\} $,
the predictive distribution, $\hat{p}_{t}^{\rho_{k}}\left(y_{t+1}\left|\left\{ y_{s},\boldsymbol{x}_{s}\right\} _{s=1}^{t},\boldsymbol{x}_{t+1}\right.\right)$,
pointwise converges to the predictive distribution, $\hat{p}_{t}^{\mathsf{m}}\left(y_{t+1}\left|\left\{ y_{s},\boldsymbol{x}_{s}\right\} _{s=1}^{t},\boldsymbol{x}_{t+1}\right.\right)$:
\[
\hat{p}_{t}^{\rho_{k}}\left(y_{t+1}\left|\left\{ y_{s},\boldsymbol{x}_{s}\right\} _{s=1}^{t},\boldsymbol{x}_{t+1}\right.\right)\rightarrow\hat{p}_{t}^{\frac{1}{v_{0}}\mathsf{m}}\left(y_{t+1}\left|\left\{ y_{s},\boldsymbol{x}_{s}\right\} _{s=1}^{t},\boldsymbol{x}_{t+1}\right.\right),\ \left(k\rightarrow\infty\right).
\]
\end{cor}

Next, we use the following lemma to show that the scale variable,
$v_{0}$, is extended Bayes. This lemma is an extension of \citet{Liang-Barron_04}. 
\begin{lem}
\label{lem:BRD_bound} Let the two priors for $v_{0}$ be $\rho\left(v_{0}\right)$
and $\varrho\left(v_{0}\right)$, and assume that the posteriors are
proper. Let $u_{s}=f\left(y_{s}\right)$, $s=1,\cdots,t+1$ be the
variable transformed value of the observations, $\left\{ y_{s}\right\} _{s=1}^{t+1}$,
based on a measurable function, $f$. Further assume that the posterior,
$\rho\left(v_{0}\right),\varrho\left(v_{0}\right)$, by the variable
transformed value is also proper. Then, the Bayes risk difference
is bounded by 
\begin{alignat*}{1}
 & B\left(\rho,\hat{p}_{t}^{\rho}\right)-B\left(\rho,\hat{p}_{t}^{\varrho}\right)\\
\leqq & \int_{U}\int_{\left(0,\infty\right]}\int_{\left(0,\infty\right]}\left\{ \log\frac{\rho\left(v_{0}\right)}{\varrho\left(v_{0}\right)}-\log\frac{\rho\left(\tilde{v}_{0}\right)}{\varrho\left(\tilde{v}_{0}\right)}\right\} \rho\left(v_{0}\right)dv_{0}\varrho\left(\left.\tilde{v}_{0}\right|u_{1},\cdots,u_{t}\right)d\tilde{v}_{0}d\mu_{U}\left(u_{1},\cdots,u_{t}\right).
\end{alignat*}
Here, $\tilde{v}_{0}$ is a copy of $v_{0}$ and $\varrho\left(\left.v_{0}\right|y_{1},\cdots,y_{t+1}\right)$
is the posterior distribution of $v_{0}$ obtained from the Kalman
smoother. 
\end{lem}

\begin{proof}
See Appendix~\ref{App:lemma3}. 
\end{proof}
From this lemma, if we approximate the Jeffreys' prior $\frac{1}{\left|v_{0}\right|}$
on the scale parameter, $v_{0}$, using a approximate sequence from
a proper prior, $\rho_{k}\left(v_{0}\right)=\frac{1}{\left|v_{0}\right|^{1+c_{k}}}$,
$\left(c_{k}\left(>0\right)\rightarrow0,\textrm{as }k\rightarrow\infty\right)$,
then we show $B\left(\rho_{k},\hat{p}_{t}^{\frac{1}{v_{0}}\mathsf{m}}\right)-B\left(\rho_{k},\hat{p}_{t}^{\rho_{k}}\right)\rightarrow0$
\citep[Theorem 2.]{Liang-Barron_04}. Thus, we have the following
proposition: 
\begin{prop}
Assume, for all processes, $\left\{ \boldsymbol{x}_{t}\right\} $,
some functional, $\varphi\left(\left\{ \boldsymbol{x}_{t}\right\} \right)$,
exists and satisfies, 
\[
\max\left(\hat{p}_{t}^{\rho_{k}}\left(y_{t+1}\left|\left\{ y_{s},\boldsymbol{x}_{s}\right\} _{s=1}^{t},\boldsymbol{x}_{t+1}\right.\right),\hat{p}_{t}^{\frac{1}{v_{0}}\mathsf{m}}\left(y_{t+1}\left|\left\{ y_{s},\boldsymbol{x}_{s}\right\} _{s=1}^{t},\boldsymbol{x}_{t+1}\right.\right)\right)<\varphi\left(\left\{ \boldsymbol{x}_{t}\right\} \right)<\infty.
\]
For the initial prior, $\rho_{k}\left(\boldsymbol{\theta}_{0}^{\prime},v_{0}\right)=N\left(0,\sigma_{k}^{2}I\right)\frac{1}{\left|v_{0}\right|^{1+c_{k}}}$,
with variance parameter, $\sigma_{k}$, where $\sigma_{k}\rightarrow\infty$
as $k\rightarrow\infty$, the following holds: 
\[
\int_{\mathbb{R}^{\otimes t}}\hat{p}_{t}^{\rho_{k}}\left(y_{t+1}\left|\left\{ y_{s},\boldsymbol{x}_{s}\right\} _{s=1}^{t},\boldsymbol{x}_{t+1}\right.\right)d\mu_{x}\rightarrow\int_{\mathbb{R}^{\otimes t}}\hat{p}_{t}^{\frac{1}{v_{0}}\mathsf{m}}\left(y_{t+1}\left|\left\{ y_{s},\boldsymbol{x}_{s}\right\} _{s=1}^{t},\boldsymbol{x}_{t+1}\right.\right)d\mu_{x},\ \left(k\rightarrow\infty\right).
\]
\end{prop}

\begin{proof}
From the Lebesgue dominant convergence theorem, 
\[
\lim_{k\rightarrow\infty}\int_{\mathbb{R}^{\otimes t}}\hat{p}_{t}^{\rho_{k}}\left(y_{t+1}\left|\left\{ y_{s},\boldsymbol{x}_{s}\right\} _{s=1}^{t},\boldsymbol{x}_{t+1}\right.\right)d\mu_{x}\rightarrow\int_{\mathbb{R}^{\otimes t}}\lim_{k\rightarrow\infty}\hat{p}_{t}^{\rho_{k}}\left(y_{t+1}\left|\left\{ y_{s},\boldsymbol{x}_{s}\right\} _{s=1}^{t},\boldsymbol{x}_{t+1}\right.\right)d\mu_{x},\ \left(k\rightarrow\infty\right),
\]
and Corollary \ref{cor:2}, the proposition holds. 
\end{proof}
From this, we see that the difference in Bayes risk is 
\[
B\left(\rho_{k},\hat{p}_{t}^{\rho_{k}}\right)\rightarrow B\left(\rho_{k},\hat{p}_{t}^{\frac{1}{v_{0}}\mathsf{m}}\right),\ \left(k\rightarrow\infty\right),
\]
thus, $\hat{p}_{t}^{\frac{1}{v_{0}}\mathsf{m}}\left(y_{t+1}\left|\left\{ y_{s},\boldsymbol{x}_{s}\right\} _{s=1}^{t},\boldsymbol{x}_{t+1}\right.\right)$
is extended Bayes.

\subsection{Details for Proposition~\ref{prop:2}\label{App:prop2detail}}

For the equation in Proposition~\ref{prop:2}, we have, 
\begin{align*}
\tilde{\varSigma}_{1} & =\left(\overline{\varSigma}_{1}^{-1}+W_{1}^{-1}\right)^{-1}, & \overline{\varSigma}_{1} & =\left(\boldsymbol{x}_{1}^{\prime}V_{1}^{-1}\boldsymbol{x}_{1}^{\prime\top}+\tilde{\varOmega}_{2}^{-1}\right)^{-1},\\
\tilde{\varOmega}_{1} & =\overline{\varSigma}_{1}+W_{1}, & \overline{\varOmega}_{1} & =\left(\boldsymbol{x}_{1}^{\prime}V_{1}^{-1}\boldsymbol{x}_{1}^{\prime\top}\right)^{-1}+\tilde{\varOmega}_{2},\\
\check{\varSigma}_{1} & =\left(\grave{\varSigma}_{1}^{-1}+W_{1}^{-1}\right)^{-1}, & \grave{\varSigma}_{1} & =\left(\boldsymbol{x}_{1}^{\prime}V_{1}^{-1}\boldsymbol{x}_{1}^{\prime\top}+\check{\varOmega}_{2}^{-1}\right)^{-1},\\
\check{\varOmega}_{1} & =\grave{\varSigma}_{1}+W_{1}, & \grave{\varOmega}_{1} & =\left(\boldsymbol{x}_{1}^{\prime}V_{1}^{-1}\boldsymbol{x}_{1}^{\prime\top}\right)^{-1}+\check{\varOmega}_{2}.
\end{align*}
The iteration of each matrix for $k=1,\cdots,n$ is 
\begin{align*}
\tilde{\varSigma}_{k} & =\left(\overline{\varSigma}_{k}^{-1}+W_{k}^{-1}\right)^{-1}, & \overline{\varSigma}_{k} & =\left(\boldsymbol{x}_{k}^{\prime}V_{k}^{-1}\boldsymbol{x}_{k}^{\prime\top}+\tilde{\varOmega}_{k+1}^{-1}\right)^{-1},\\
\tilde{\varOmega}_{k} & =\overline{\varSigma}_{k}+W_{k}, & \overline{\varOmega}_{k} & =\left(\boldsymbol{x}_{k}^{\prime}V_{k}^{-1}\boldsymbol{x}_{k}^{\prime\top}\right)^{-1}+\tilde{\varOmega}_{k+1},
\end{align*}
\begin{alignat*}{1}
\tilde{\varSigma}_{n+1} & =\left(\boldsymbol{x}_{n+1}^{\prime}V_{n+1}^{-1}\boldsymbol{x}_{n+1}^{\prime\top}+W_{n+1}^{-1}\right)^{-1},\\
\tilde{\varOmega}_{n+1} & =\left(\boldsymbol{x}_{n+1}^{\prime}V_{n+1}^{-1}\boldsymbol{x}_{n+1}^{\prime\top}\right)^{-1}+W_{n+1},
\end{alignat*}
and for $k=1,\cdots,n-1$ is 
\begin{align*}
\check{\varSigma}_{k} & =\left(\grave{\varSigma}_{k}^{-1}+W_{k}^{-1}\right)^{-1}, & \grave{\varSigma}_{k} & =\left(\boldsymbol{x}_{k}^{\prime}V_{k}^{-1}\boldsymbol{x}_{k}^{\prime\top}+\check{\varOmega}_{k+1}^{-1}\right)^{-1},\\
\check{\varOmega}_{k} & =\grave{\varSigma}_{k}+W_{k} & \grave{\varOmega}_{k}, & =\left(\boldsymbol{x}_{k}^{\prime}V_{k}^{-1}\boldsymbol{x}_{k}^{\prime\top}\right)^{-1}+\check{\varOmega}_{k+1},
\end{align*}
\begin{alignat*}{1}
\check{\varSigma}_{n} & =\left(\boldsymbol{x}_{n}^{\prime}V_{n}^{-1}\boldsymbol{x}_{n}^{\prime\top}+W_{n}^{-1}\right)^{-1},\\
\check{\varOmega}_{n} & =\left(\boldsymbol{x}_{n}^{\prime}V_{n}^{-1}\boldsymbol{x}_{n}^{\prime\top}\right)^{-1}+W_{n}.
\end{alignat*}

\subsection{Proof of Proposition \ref{prop:2}\label{App:prop2}}

For the random walk DLM, 
\begin{alignat*}{2}
y_{t} & =\boldsymbol{x}_{t}^{\prime\top}\boldsymbol{\theta}_{t}^{\prime}+v_{t}, &  & v_{t}\sim N\left(0,V_{t}\right),\\
\boldsymbol{\theta}_{t}^{\prime} & =\boldsymbol{\theta}_{t-1}^{\prime}+w_{t}, &  & w_{t}\sim N\left(0,W_{t}\right),
\end{alignat*}
the conditional density function for the observation and state equation
is given as 
\begin{alignat*}{1}
y_{t}\left|\boldsymbol{\theta}_{t}^{\prime}\right. & \sim N\left(\boldsymbol{x}_{t}^{\prime\top}\boldsymbol{\theta}_{t}^{\prime},V_{t}\right)=\frac{1}{\sqrt{2\pi\left|V_{t}\right|}}\exp\left\{ -\left(y_{t}-\boldsymbol{x}_{t}^{\prime\top}\boldsymbol{\theta}_{t}^{\prime}\right)^{\top}V_{t}^{-1}\left(y_{t}-\boldsymbol{x}_{t}^{\prime\top}\boldsymbol{\theta}_{t}^{\prime}\right)\right\} ,\\
\boldsymbol{\theta}_{t}^{\prime}\left|\boldsymbol{\theta}_{t-1}^{\prime}\right. & \sim N\left(\boldsymbol{\theta}_{t-1}^{\prime},W_{t}\right)=\frac{1}{\sqrt{2\pi\left|W_{t}\right|}}\exp\left\{ -\left(\boldsymbol{\theta}_{t}^{\prime}-\boldsymbol{\theta}_{t-1}^{\prime}\right)^{\top}W_{t}^{-1}\left(\boldsymbol{\theta}_{t}^{\prime}-\boldsymbol{\theta}_{t-1}^{\prime}\right)\right\} .
\end{alignat*}
If we set $\hat{\boldsymbol{\theta}}_{n}=\left(\boldsymbol{x}_{t}^{\prime}\boldsymbol{x}_{t}^{\prime\top}\right)^{-1}\boldsymbol{x}_{t}^{\prime}y_{n}$
the conditional density function of the observation equation can be
transformed to 
\begin{alignat*}{1}
 & \frac{1}{\sqrt{2\pi\left|V_{t}\right|}}\exp\left\{ -\left(y_{t}-\boldsymbol{x}_{t}^{\prime\top}\boldsymbol{\theta}_{t}^{\prime}\right)^{\top}V_{t}^{-1}\left(y_{t}-\boldsymbol{x}_{t}^{\prime\top}\boldsymbol{\theta}_{t}^{\prime}\right)\right\} \\
= & \frac{1}{\sqrt{2\pi\left|V_{t}\right|}}\exp\left\{ -\left(y_{t}-\boldsymbol{x}_{t}^{\prime\top}\hat{\boldsymbol{\theta}}_{t}\right)^{\top}V_{t}^{-1}\left(y_{t}-\boldsymbol{x}_{t}^{\prime\top}\hat{\boldsymbol{\theta}}_{t}\right)\right\} \\
 & \times\exp\left\{ -\left(\hat{\boldsymbol{\theta}}_{t}-\boldsymbol{\theta}_{t}^{\prime}\right)^{\top}\boldsymbol{x}_{t}^{\prime}V_{t}^{-1}\boldsymbol{x}_{t}^{\prime\top}\left(\hat{\boldsymbol{\theta}}_{t}-\boldsymbol{\theta}_{t}^{\prime}\right)\right\} .
\end{alignat*}

The predictive distribution, $q_{t}^{\rho}\left(y_{t+1}\left|\left\{ y_{s},\boldsymbol{x}_{s}\right\} _{s=1}^{t},\boldsymbol{x}_{t+1}\right.\right)$,
is given as, 
\begin{equation}
q_{t}^{\rho}\left(y_{t+1}\left|\left\{ y_{s},\boldsymbol{x}_{s}\right\} _{s=1}^{t},\boldsymbol{x}_{t+1}\right.\right)=\frac{\int_{\left(\mathbb{R}^{J}\right)^{t+1}}\beta_{t+1}\left(\boldsymbol{\theta}^{\prime},y\right)d\mu_{\theta}^{\rho}\left(\boldsymbol{\theta}_{0}^{\prime},\cdots,\boldsymbol{\theta}_{t+1}^{\prime}\right)}{\int_{\left(\mathbb{R}^{J}\right)^{t}}\beta_{t}\left(\boldsymbol{\theta}^{\prime},y\right)d\mu_{\theta}^{\rho}\left(\boldsymbol{\theta}_{0}^{\prime},\cdots,\boldsymbol{\theta}_{t}^{\prime}\right)}.\label{eq:Apnd1}
\end{equation}
Here, 
\begin{alignat*}{1}
\beta_{t}\left(\boldsymbol{\theta}^{\prime},y\right) & =\exp\left\{ \sum_{k=1}^{t}\left\langle \boldsymbol{\theta}_{k}^{\prime},\boldsymbol{x}_{k}^{\prime}\right\rangle y_{k}-\frac{1}{2}\sum_{k=1}^{n}\left(\left\langle \boldsymbol{\theta}_{k}^{\prime},\boldsymbol{x}_{k}^{\prime}\right\rangle \right)^{2}\right\} \\
 & =\prod_{k=1}^{n}\exp\left\{ \left\langle \boldsymbol{\theta}_{k}^{\prime},\boldsymbol{x}_{k}^{\prime}\right\rangle y_{k}-\frac{1}{2}\left\langle \boldsymbol{\theta}_{k}^{\prime},\boldsymbol{x}_{k}^{\prime}\right\rangle ^{2}\right\} .
\end{alignat*}
We first rewrite the generalized Bayes formula to the sequential formula.
Given, 
\begin{align*}
g_{t}\left(y_{t}\left|\boldsymbol{\theta}_{t}^{\prime}\right.\right) & =N\left(\boldsymbol{x}_{t}^{\prime\top}\boldsymbol{\theta}_{t}^{\prime},V_{t}\right),\\
P_{t}\left(\boldsymbol{\theta}_{t}^{\prime}\left|\boldsymbol{\theta}_{t-1}^{\prime}\right.\right) & =N\left(\boldsymbol{\theta}_{t-1}^{\prime},W_{t}\right),
\end{align*}
the numerator of the left hand side of eq.~\eqref{eq:Apnd1} can
be written as 
\begin{alignat*}{1}
 & \int\left[\cdots\int\left[g_{t+1}\left(y_{t+1}\left|\boldsymbol{\theta}_{t+1}^{\prime}\right.\right)P_{t+1}\left(\boldsymbol{\theta}_{t+1}^{\prime}\left|\boldsymbol{\theta}_{t}^{\prime}\right.\right)d\boldsymbol{\theta}_{t+1}^{\prime}\right]\right.\\
 & \left.\times g_{t}\left(y_{t}\left|\boldsymbol{\theta}_{t}^{\prime}\right.\right)P_{t}\left(\boldsymbol{\theta}_{t}^{\prime}\left|\boldsymbol{\theta}_{t-1}^{\prime}\right.\right)d\boldsymbol{\theta}_{t}^{\prime}\right]\\
 & \left.\times g_{t-1}\left(y_{t-1}\left|\boldsymbol{\theta}_{t-1}^{\prime}\right.\right)P_{t+1}\left(\boldsymbol{\theta}_{t-1}^{\prime}\left|\boldsymbol{\theta}_{t-2}^{\prime}\right.\right)d\boldsymbol{\theta}_{t-1}^{\prime}\right]\\
 & \vdots\\
 & \left.\times g_{1}\left(y_{1}\left|\boldsymbol{\theta}_{1}^{\prime}\right.\right)P_{t}\left(\boldsymbol{\theta}_{1}^{\prime}\left|\boldsymbol{\theta}_{0}^{\prime}\right.\right)d\boldsymbol{\theta}_{1}^{\prime}\right]\\
 & \left.\times\rho\left(\boldsymbol{\theta}_{0}^{\prime}\right)d\boldsymbol{\theta}_{0}^{\prime}\right].
\end{alignat*}
From the law of iteration of conditional expectations, we integrate
from $t+1$ backwards to sequentially remove the conditionals. Note,
however, that we do not integrate over quantities that are not necessary
for the proof.

First, for $\int g_{t+1}\left(y_{t+1}\left|\boldsymbol{\theta}_{t+1}^{\prime}\right.\right)P_{t+1}\left(\boldsymbol{\theta}_{t+1}^{\prime}\left|\boldsymbol{\theta}_{t}^{\prime}\right.\right)d\boldsymbol{\theta}_{t+1}^{\prime}$,
we have, 
\begin{alignat*}{1}
 & g_{t+1}\left(y_{t+1}\left|\boldsymbol{\theta}_{t+1}^{\prime}\right.\right)P_{t+1}\left(\boldsymbol{\theta}_{t+1}^{\prime}\left|\boldsymbol{\theta}_{t}^{\prime}\right.\right)\\
= & \frac{1}{\sqrt{2\pi\left|V_{t}\right|}}\exp\left\{ -\left(y_{t+1}-\boldsymbol{x}_{t+1}^{\prime\top}\hat{\boldsymbol{\theta}}_{t+1}\right)^{\top}V_{t+1}^{-1}\left(y_{t+1}-\boldsymbol{x}_{t+1}^{\prime\top}\hat{\boldsymbol{\theta}}_{t+1}\right)\right\} \\
 & \times\exp\left\{ -\left(\hat{\boldsymbol{\theta}}_{t+1}-\boldsymbol{\theta}_{t+1}^{\prime}\right)^{\top}\boldsymbol{x}_{t+1}^{\prime}V_{t+1}^{-1}\boldsymbol{x}_{t+1}^{\prime\top}\left(\hat{\boldsymbol{\theta}}_{t+1}-\boldsymbol{\theta}_{t+1}^{\prime}\right)\right\} \\
 & \times\frac{1}{\sqrt{2\pi\left|W_{t+1}\right|}}\exp\left\{ -\left(\boldsymbol{\theta}_{t+1}^{\prime}-\boldsymbol{\theta}_{t}^{\prime}\right)^{\top}W_{t+1}^{-1}\left(\boldsymbol{\theta}_{t+1}^{\prime}-\boldsymbol{\theta}_{t}^{\prime}\right)\right\} .
\end{alignat*}
Completing the square inside the exponent, we have 
\begin{alignat*}{1}
 & \left(\hat{\boldsymbol{\theta}}_{t+1}-\boldsymbol{\theta}_{t+1}^{\prime}\right)^{\top}\boldsymbol{x}_{t+1}^{\prime}V_{t+1}^{-1}\boldsymbol{x}_{t+1}^{\prime\top}\left(\hat{\boldsymbol{\theta}}_{t+1}-\boldsymbol{\theta}_{t+1}^{\prime}\right)+\left(\boldsymbol{\theta}_{t+1}^{\prime}-\boldsymbol{\theta}_{t}^{\prime}\right)^{\top}W_{t+1}^{-1}\left(\boldsymbol{\theta}_{t+1}^{\prime}-\boldsymbol{\theta}_{t}^{\prime}\right)\\
= & \left(\boldsymbol{\theta}_{t+1}^{\prime}-\tilde{\boldsymbol{\theta}}_{t+1}\right)^{\top}\left(\tilde{\varSigma}_{t+1}\right)^{-1}\left(\boldsymbol{\theta}_{t+1}^{\prime}-\tilde{\boldsymbol{\theta}}_{t+1}\right)+\left(\hat{\boldsymbol{\theta}}_{t+1}-\boldsymbol{\theta}_{t}^{\prime}\right)^{\top}\left(\tilde{\varOmega}_{t+1}\right)^{-1}\left(\hat{\boldsymbol{\theta}}_{t+1}-\boldsymbol{\theta}_{t}^{\prime}\right),
\end{alignat*}
where, 
\begin{alignat*}{1}
\tilde{\boldsymbol{\theta}}_{t+1} & =\left(\boldsymbol{x}_{t+1}^{\prime}V_{t+1}^{-1}\boldsymbol{x}_{t+1}^{\prime\top}+W_{t+1}^{-1}\right)^{-1}\left(\boldsymbol{x}_{t+1}^{\prime}V_{n+1}^{-1}\boldsymbol{x}_{t+1}^{\prime\top}\hat{\boldsymbol{\theta}}_{t+1}+W_{t+1}^{-1}\boldsymbol{\theta}_{t}^{\prime}\right),\\
\tilde{\varSigma}_{t+1} & =\left(\boldsymbol{x}_{t+1}^{\prime}V_{t+1}^{-1}\boldsymbol{x}_{t+1}^{\prime\top}+W_{t+1}^{-1}\right)^{-1},\\
\tilde{\varOmega}_{t+1} & =\left(\boldsymbol{x}_{t+1}^{\prime}V_{t+1}^{-1}\boldsymbol{x}_{t+1}^{\prime\top}\right)^{-1}+W_{t+1}.
\end{alignat*}
Since $\exp\left\{ -\left(\hat{\boldsymbol{\theta}}_{t+1}-\boldsymbol{\theta}_{t}^{\prime}\right)^{\top}\left(\tilde{\varOmega}_{t+1}\right)^{-1}\left(\hat{\boldsymbol{\theta}}_{t+1}-\boldsymbol{\theta}_{t}^{\prime}\right)\right\} $
does not affect the integration of $d\boldsymbol{\theta}_{t+1}^{\prime}$,
we move it one step to the past.

Next, we integrate the portion concerning $d\boldsymbol{\theta}_{t}^{\prime}$:
\[
\int\exp\left\{ -\left(\hat{\boldsymbol{\theta}}_{t+1}-\boldsymbol{\theta}_{t}^{\prime}\right)^{\top}\left(\tilde{\varOmega}_{t+1}\right)^{-1}\left(\hat{\boldsymbol{\theta}}_{t+1}-\boldsymbol{\theta}_{t}^{\prime}\right)\right\} g_{t}\left(y_{t}\left|\boldsymbol{\theta}_{t}^{\prime}\right.\right)P_{t}\left(\boldsymbol{\theta}_{t}^{\prime}\left|\boldsymbol{\theta}_{t-1}^{\prime}\right.\right)d\boldsymbol{\theta}_{t}^{\prime}.
\]
The equation inside the exponent is 
\begin{alignat*}{1}
 & \left(\hat{\boldsymbol{\theta}}_{t+1}-\boldsymbol{\theta}_{t}^{\prime}\right)^{\top}\left(\tilde{\varOmega}_{t+1}\right)^{-1}\left(\hat{\boldsymbol{\theta}}_{t+1}-\boldsymbol{\theta}_{t}^{\prime}\right)+\left(y_{t}-\boldsymbol{x}_{t}^{\prime\top}\hat{\boldsymbol{\theta}}_{t}\right)^{\top}V_{t}^{-1}\left(y_{t}-\boldsymbol{x}_{t}^{\prime\top}\hat{\boldsymbol{\theta}}_{t}\right)\\
 & +\left(\hat{\boldsymbol{\theta}}_{t}-\boldsymbol{\theta}_{t}^{\prime}\right)^{\top}\boldsymbol{x}_{t}^{\prime}V_{t}^{-1}\boldsymbol{x}_{t}^{\prime\top}\left(\hat{\boldsymbol{\theta}}_{t}-\boldsymbol{\theta}_{t}^{\prime}\right)+\left(\boldsymbol{\theta}_{t}^{\prime}-\boldsymbol{\theta}_{t-1}^{\prime}\right)^{\top}W_{t}^{-1}\left(\boldsymbol{\theta}_{t}^{\prime}-\boldsymbol{\theta}_{t-1}^{\prime}\right).
\end{alignat*}
Then, we have, 
\begin{alignat*}{1}
 & \left(\hat{\boldsymbol{\theta}}_{t+1}-\boldsymbol{\theta}_{t}^{\prime}\right)^{\top}\left(\tilde{\varOmega}_{t+1}\right)^{-1}\left(\hat{\boldsymbol{\theta}}_{t+1}-\boldsymbol{\theta}_{t}^{\prime}\right)+\left(\hat{\boldsymbol{\theta}}_{t}-\boldsymbol{\theta}_{t}^{\prime}\right)^{\top}\boldsymbol{x}_{t}^{\prime}V_{t}^{-1}\boldsymbol{x}_{t}^{\prime\top}\left(\hat{\boldsymbol{\theta}}_{t}-\boldsymbol{\theta}_{t}^{\prime}\right)\\
= & \left(\boldsymbol{\theta}_{t}^{\prime}-\overline{\boldsymbol{\theta}}_{t}\right)^{\top}\left(\overline{\varSigma}_{t}\right)^{-1}\left(\boldsymbol{\theta}_{t}^{\prime}-\overline{\boldsymbol{\theta}}_{t}\right)+\left(\hat{\boldsymbol{\theta}}_{t+1}-\hat{\boldsymbol{\theta}}_{t}\right)^{\top}\left(\overline{\varOmega}_{t}\right)^{-1}\left(\hat{\boldsymbol{\theta}}_{t+1}-\hat{\boldsymbol{\theta}}_{t}\right),
\end{alignat*}
where, 
\begin{alignat*}{1}
\overline{\boldsymbol{\theta}}_{t} & =\left(\boldsymbol{x}_{t}^{\prime}V_{t}^{-1}\boldsymbol{x}_{t}^{\prime\top}+\tilde{\varOmega}_{t+1}^{-1}\right)^{-1}\left(\boldsymbol{x}_{t}^{\prime}V_{t}^{-1}\boldsymbol{x}_{t}^{\prime\top}\hat{\boldsymbol{\theta}}_{t}+\tilde{\varOmega}_{t+1}^{-1}\hat{\boldsymbol{\theta}}_{t+1}\right),\\
\overline{\varSigma}_{t} & =\left(\boldsymbol{x}_{t}^{\prime}V_{t}^{-1}\boldsymbol{x}_{t}^{\prime\top}+\tilde{\varOmega}_{t+1}^{-1}\right)^{-1},\\
\overline{\varOmega}_{t} & =\left(\boldsymbol{x}_{t}^{\prime}V_{t}^{-1}\boldsymbol{x}_{t}^{\prime\top}\right)^{-1}+\tilde{\varOmega}_{t+1}^{-1}.
\end{alignat*}
Next, we have, 
\begin{alignat*}{1}
 & \left(\boldsymbol{\theta}_{t}^{\prime}-\overline{\boldsymbol{\theta}}_{t}\right)^{\top}\left(\overline{\varSigma}_{t}\right)^{-1}\left(\boldsymbol{\theta}_{t}^{\prime}-\overline{\boldsymbol{\theta}}_{t}\right)+\left(\boldsymbol{\theta}_{t}^{\prime}-\boldsymbol{\theta}_{t-1}^{\prime}\right)^{\top}W_{t}^{-1}\left(\boldsymbol{\theta}_{t}^{\prime}-\boldsymbol{\theta}_{t-1}^{\prime}\right)\\
= & \left(\boldsymbol{\theta}_{t}^{\prime}-\tilde{\boldsymbol{\theta}}_{t}\right)^{\top}\left(\tilde{\varSigma}_{t}\right)^{-1}\left(\boldsymbol{\theta}_{t}^{\prime}-\tilde{\boldsymbol{\theta}}_{t}\right)+\left(\overline{\boldsymbol{\theta}}_{t}-\boldsymbol{\theta}_{t-1}^{\prime}\right)^{\top}\left(\tilde{\varOmega}_{t}\right)^{-1}\left(\overline{\boldsymbol{\theta}}_{t}-\boldsymbol{\theta}_{t-1}^{\prime}\right),
\end{alignat*}
where, 
\begin{alignat*}{1}
\tilde{\boldsymbol{\theta}}_{t} & =\left(\overline{\varSigma}_{t}^{-1}+W_{t}^{-1}\right)^{-1}\left(\overline{\varSigma}_{t}^{-1}\overline{\boldsymbol{\theta}}_{t}+W_{t}^{-1}\boldsymbol{\theta}_{t-1}^{\prime}\right),\\
\tilde{\varSigma}_{t} & =\left(\overline{\varSigma}_{t}^{-1}+W_{t}^{-1}\right)^{-1},\\
\tilde{\varOmega}_{t} & =\overline{\varSigma}_{t}+W_{t}.
\end{alignat*}
Since $\exp\left\{ -\left(\overline{\boldsymbol{\theta}}_{t}-\boldsymbol{\theta}_{t-1}^{\prime}\right)^{\top}\left(\tilde{\varOmega}_{t}\right)^{-1}\left(\overline{\boldsymbol{\theta}}_{t}-\boldsymbol{\theta}_{t-1}^{\prime}\right)\right\} $
does not affect the integration of $d\boldsymbol{\theta}_{t}^{\prime}$,
we move it one step to the past.

Likewise, we sequentially integrate backwards. Then, for the integration
of $d\boldsymbol{\theta}_{0}^{\prime}$, we have, 
\[
\int\exp\left\{ -\left(\overline{\boldsymbol{\theta}}_{1}-\boldsymbol{\theta}_{0}^{\prime}\right)^{\top}\left(\tilde{\varOmega}_{1}\right)^{-1}\left(\overline{\boldsymbol{\theta}}_{1}-\boldsymbol{\theta}_{0}^{\prime}\right)\right\} \rho\left(\boldsymbol{\theta}_{0}^{\prime}\right)d\boldsymbol{\theta}_{0}^{\prime}.
\]
If we take a Lebesgue measure for $\rho$, this integrated value is
$\sqrt{\left(2\pi\left|\tilde{\varOmega}_{1}\right|\right)^{J}}$.
If $\rho=\mathcal{N}\left(0,W_{0}\right)$, then inside the exponent
is 
\begin{alignat*}{1}
 & \left(\overline{\boldsymbol{\theta}}_{1}-\boldsymbol{\theta}_{0}^{\prime}\right)^{\top}\left(\tilde{\varOmega}_{1}\right)^{-1}\left(\overline{\boldsymbol{\theta}}_{1}-\boldsymbol{\theta}_{0}^{\prime}\right)+\left(\boldsymbol{\theta}_{0}^{\prime}\right)^{\top}W_{0}^{-1}\left(\boldsymbol{\theta}_{0}^{\prime}\right)\\
= & \left(\boldsymbol{\theta}_{0}^{\prime}-\tilde{\boldsymbol{\theta}}_{0}\right)^{\top}\left(\tilde{\varSigma}_{0}\right)^{-1}\left(\boldsymbol{\theta}_{0}^{\prime}-\tilde{\boldsymbol{\theta}}_{0}\right)+\left(\overline{\boldsymbol{\theta}}_{1}\right)^{\top}\left(\tilde{\varOmega}_{0}\right)^{-1}\left(\overline{\boldsymbol{\theta}}_{1}\right),
\end{alignat*}
where, 
\begin{alignat*}{1}
\tilde{\boldsymbol{\theta}}_{0} & =\left(\tilde{\varOmega}_{1}^{-1}+W_{0}^{-1}\right)^{-1}\tilde{\varOmega}_{1}^{-1}\overline{\boldsymbol{\theta}}_{1},\\
\tilde{\varSigma}_{0} & =\left(\tilde{\varOmega}_{1}^{-1}+W_{0}^{-1}\right)^{-1},\\
\tilde{\varOmega}_{0} & =\tilde{\varOmega}_{1}+W_{0}.
\end{alignat*}
Therefore, the integrated value is 
\[
\exp\left\{ -\left(\overline{\boldsymbol{\theta}}_{1}\right)^{\top}\left(\tilde{\varOmega}_{0}\right)^{-1}\left(\overline{\boldsymbol{\theta}}_{1}\right)\right\} \sqrt{\left(2\pi\left|\tilde{\varSigma}_{0}\right|\right)^{J}}.
\]
From this, the numerator of eq.~\eqref{eq:Apnd1} can be expressed
as 
\begin{align}
 & \int_{\left(\mathbb{R}^{J}\right)^{t+1}}\beta_{t+1}\left(\boldsymbol{\theta}^{\prime},y\right)d\mu_{\theta}^{\rho}\left(\boldsymbol{\theta}_{0}^{\prime},\cdots,\boldsymbol{\theta}_{t+1}^{\prime}\right)\nonumber \\
= & \int_{\left(\mathbb{R}^{J}\right)^{t+1}}\beta_{t+1}\left(\boldsymbol{\theta}^{\prime},y\right)d\mu_{\theta}^{\mathsf{m}}\left(\boldsymbol{\theta}_{0}^{\prime},\cdots,\boldsymbol{\theta}_{t+1}^{\prime}\right)\times\frac{\exp\left\{ -\left(\overline{\boldsymbol{\theta}}_{1}\right)^{\top}\left(\tilde{\varOmega}_{0}\right)^{-1}\left(\overline{\boldsymbol{\theta}}_{1}\right)\right\} \sqrt{\left(2\pi\left|\tilde{\varSigma}_{0}\right|\right)^{J}}}{\sqrt{\left(2\pi\left|\tilde{\varOmega}_{1}\right|\right)^{J}}}.\label{eq:KFnume}
\end{align}

We now integrate the denominator of eq.~\eqref{eq:Apnd1}. This integration
is not as simple as offsetting the subscript from the results of the
numerator, as the multiple integration of the denominator concerns
one less variable than the integration of the numerator.

For the integration with regard to $d\boldsymbol{\theta}_{t}^{\prime}$,
\[
\int g_{t}\left(y_{t}\left|\boldsymbol{\theta}_{t}^{\prime}\right.\right)P_{t}\left(\boldsymbol{\theta}_{t}^{\prime}\left|\boldsymbol{\theta}_{t-1}^{\prime}\right.\right)d\boldsymbol{\theta}_{t}^{\prime},
\]
the equation inside the exponent is 
\begin{alignat*}{1}
 & \left(y_{t}-\boldsymbol{x}_{t}^{\prime\top}\hat{\boldsymbol{\theta}}_{t}\right)^{\top}V_{t}^{-1}\left(y_{t}-\boldsymbol{x}_{t}^{\prime\top}\hat{\boldsymbol{\theta}}_{t}\right)\\
 & +\left(\hat{\boldsymbol{\theta}}_{t}-\boldsymbol{\theta}_{t}^{\prime}\right)^{\top}\boldsymbol{x}_{t}^{\prime}V_{t}^{-1}\boldsymbol{x}_{t}^{\prime\top}\left(\hat{\boldsymbol{\theta}}_{t}-\boldsymbol{\theta}_{t}^{\prime}\right)+\left(\boldsymbol{\theta}_{t}^{\prime}-\boldsymbol{\theta}_{t-1}^{\prime}\right)^{\top}W_{t}^{-1}\left(\boldsymbol{\theta}_{t}^{\prime}-\boldsymbol{\theta}_{t-1}^{\prime}\right)\\
= & \left(\boldsymbol{\theta}_{t}^{\prime}-\check{\boldsymbol{\theta}}_{t}\right)^{\top}\left(\check{\varSigma}_{t}\right)^{-1}\left(\boldsymbol{\theta}_{t}^{\prime}-\check{\boldsymbol{\theta}}_{t}\right)+\left(\hat{\boldsymbol{\theta}}_{t}-\boldsymbol{\theta}_{t-1}^{\prime}\right)^{\top}\left(\check{\varOmega}_{t}\right)^{-1}\left(\hat{\boldsymbol{\theta}}_{t}-\boldsymbol{\theta}_{t-1}^{\prime}\right),
\end{alignat*}
where, 
\begin{alignat*}{1}
\check{\boldsymbol{\theta}}_{t} & =\left(\boldsymbol{x}_{t}^{\prime}V_{t}^{-1}\boldsymbol{x}_{t}^{\prime\top}+W_{t}^{-1}\right)^{-1}\left(\boldsymbol{x}_{t}^{\prime}V_{t}^{-1}\boldsymbol{x}_{t}^{\prime\top}\hat{\boldsymbol{\theta}}_{t}+W_{t}^{-1}\boldsymbol{\theta}_{t-1}^{\prime}\right),\\
\check{\varSigma}_{t} & =\left(\boldsymbol{x}_{t}^{\prime}V_{t}^{-1}\boldsymbol{x}_{t}^{\prime\top}+W_{t}^{-1}\right)^{-1},\\
\check{\varOmega}_{t} & =\left(\boldsymbol{x}_{t}^{\prime}V_{t}^{-1}\boldsymbol{x}_{t}^{\prime\top}\right)^{-1}+W_{t}.
\end{alignat*}
We move $\exp\left\{ -\left(\hat{\boldsymbol{\theta}}_{t}-\boldsymbol{\theta}_{t-1}^{\prime}\right)^{\top}\left(\check{\varOmega}_{t}\right)^{-1}\left(\hat{\boldsymbol{\theta}}_{t}-\boldsymbol{\theta}_{t-1}^{\prime}\right)\right\} $
back one step to the past.

Next, for the integration concerning $d\boldsymbol{\theta}_{t-1}^{\prime}$,
\[
\int\exp\left\{ -\left(\hat{\boldsymbol{\theta}}_{t}-\boldsymbol{\theta}_{t-1}^{\prime}\right)^{\top}\left(\check{\varOmega}_{t}\right)^{-1}\left(\hat{\boldsymbol{\theta}}_{t}-\boldsymbol{\theta}_{t-1}^{\prime}\right)\right\} g_{t-1}\left(y_{t-1}\left|\boldsymbol{\theta}_{t-1}^{\prime}\right.\right)P_{t-1}\left(\boldsymbol{\theta}_{t-1}^{\prime}\left|\boldsymbol{\theta}_{t-2}^{\prime}\right.\right)d\boldsymbol{\theta}_{t-1}^{\prime},
\]
the equation inside the exponent is 
\begin{alignat*}{1}
 & \left(\hat{\boldsymbol{\theta}}_{t}-\boldsymbol{\theta}_{t-1}^{\prime}\right)^{\top}\left(\check{\varOmega}_{t}\right)^{-1}\left(\hat{\boldsymbol{\theta}}_{t}-\boldsymbol{\theta}_{t-1}^{\prime}\right)+\left(y_{t-1}-\boldsymbol{x}_{t-1}^{\prime\top}\hat{\boldsymbol{\theta}}_{t-1}\right)^{\top}V_{t-1}^{-1}\left(y_{t-1}-\boldsymbol{x}_{t-1}^{\prime\top}\hat{\boldsymbol{\theta}}_{t-1}\right)\\
 & +\left(\hat{\boldsymbol{\theta}}_{t-1}-\boldsymbol{\theta}_{t-1}^{\prime}\right)^{\top}\boldsymbol{x}_{t-1}^{\prime}V_{t-1}^{-1}\boldsymbol{x}_{t-1}^{\prime\top}\left(\hat{\boldsymbol{\theta}}_{t-1}-\boldsymbol{\theta}_{t-1}^{\prime}\right)+\left(\boldsymbol{\theta}_{t-1}^{\prime}-\boldsymbol{\theta}_{t-2}^{\prime}\right)^{\top}W_{t-1}^{-1}\left(\boldsymbol{\theta}_{t-1}^{\prime}-\boldsymbol{\theta}_{t-2}^{\prime}\right).
\end{alignat*}
Then, we have, 
\begin{alignat*}{1}
 & \left(\hat{\boldsymbol{\theta}}_{t}-\boldsymbol{\theta}_{t-1}^{\prime}\right)^{\top}\left(\check{\varOmega}_{t}\right)^{-1}\left(\hat{\boldsymbol{\theta}}_{t}-\boldsymbol{\theta}_{t-1}^{\prime}\right)+\left(\hat{\boldsymbol{\theta}}_{t-1}-\boldsymbol{\theta}_{t-1}^{\prime}\right)^{\top}\boldsymbol{x}_{t-1}^{\prime}V_{t-1}^{-1}\boldsymbol{x}_{t-1}^{\prime\top}\left(\hat{\boldsymbol{\theta}}_{t-1}-\boldsymbol{\theta}_{t-1}^{\prime}\right)\\
= & \left(\boldsymbol{\theta}_{t-1}^{\prime}-\grave{\boldsymbol{\theta}}_{t-1}\right)^{\top}\left(\grave{\varSigma}_{t-1}\right)^{-1}\left(\boldsymbol{\theta}_{t-1}^{\prime}-\grave{\boldsymbol{\theta}}_{t-1}\right)+\left(\hat{\boldsymbol{\theta}}_{t-1}-\hat{\boldsymbol{\theta}}_{t}\right)^{\top}\left(\grave{\varOmega}_{t}\right)^{-1}\left(\hat{\boldsymbol{\theta}}_{t-1}-\hat{\boldsymbol{\theta}}_{t}\right),
\end{alignat*}
where, 
\begin{alignat*}{1}
\grave{\boldsymbol{\theta}}_{t-1} & =\left(\boldsymbol{x}_{t-1}^{\prime}V_{t-1}^{-1}\boldsymbol{x}_{t-1}^{\prime\top}+\check{\varOmega}_{t}^{-1}\right)^{-1}\left(\boldsymbol{x}_{t-1}^{\prime}V_{t-1}^{-1}\boldsymbol{x}_{t-1}^{\prime\top}\hat{\boldsymbol{\theta}}_{t-1}+\check{\varOmega}_{t}^{-1}\hat{\boldsymbol{\theta}}_{t}\right),\\
\grave{\varSigma}_{t-1} & =\left(\boldsymbol{x}_{t-1}^{\prime}V_{t-1}^{-1}\boldsymbol{x}_{t-1}^{\prime\top}+\check{\varOmega}_{t}^{-1}\right)^{-1},\\
\grave{\varOmega}_{t-1} & =\left(\boldsymbol{x}_{t-1}^{\prime}V_{t-1}^{-1}\boldsymbol{x}_{t-1}^{\prime\top}\right)^{-1}+\check{\varOmega}_{t}.
\end{alignat*}
Next, we have, 
\begin{alignat*}{1}
 & \left(\boldsymbol{\theta}_{t-1}^{\prime}-\grave{\boldsymbol{\theta}}_{t-1}\right)^{\top}\left(\grave{\varSigma}_{t-1}\right)^{-1}\left(\boldsymbol{\theta}_{t-1}^{\prime}-\grave{\boldsymbol{\theta}}_{t-1}\right)+\left(\boldsymbol{\theta}_{t-1}^{\prime}-\boldsymbol{\theta}_{t-2}^{\prime}\right)^{\top}W_{t-1}^{-1}\left(\boldsymbol{\theta}_{t-1}^{\prime}-\boldsymbol{\theta}_{t-2}^{\prime}\right)\\
= & \left(\boldsymbol{\theta}_{t-1}^{\prime}-\check{\boldsymbol{\theta}}_{t-1}\right)^{\top}\left(\check{\varSigma}_{t-1}\right)^{-1}\left(\boldsymbol{\theta}_{t-1}^{\prime}-\check{\boldsymbol{\theta}}_{t-1}\right)+\left(\grave{\boldsymbol{\theta}}_{t-1}-\boldsymbol{\theta}_{t-2}^{\prime}\right)^{\top}\check{\varOmega}_{t-1}^{-1}\left(\grave{\boldsymbol{\theta}}_{t-1}-\boldsymbol{\theta}_{t-2}^{\prime}\right),
\end{alignat*}
where, 
\begin{alignat*}{1}
\check{\boldsymbol{\theta}}_{t-1} & =\left(\grave{\varSigma}_{t-1}^{-1}+W_{t-1}^{-1}\right)^{-1}\left(\grave{\varSigma}_{t-1}^{-1}\grave{\boldsymbol{\theta}}_{t-1}+W_{t-1}^{-1}\boldsymbol{\theta}_{t-2}^{\prime}\right),\\
\check{\varSigma}_{t-1} & =\left(\grave{\varSigma}_{t-1}^{-1}+W_{t-1}^{-1}\right)^{-1},\\
\check{\varOmega}_{t-1} & =\grave{\varSigma}_{t-1}+W_{t-1}.
\end{alignat*}
We move $\exp\left\{ -\left(\grave{\boldsymbol{\theta}}_{t-1}-\boldsymbol{\theta}_{t-2}^{\prime}\right)^{\top}\check{\varOmega}_{t-1}^{-1}\left(\grave{\boldsymbol{\theta}}_{t-1}-\boldsymbol{\theta}_{t-2}^{\prime}\right)\right\} $
back one step to the past.

Likewise, we sequentially integrate backwards. Then, for the integration
of $d\boldsymbol{\theta}_{0}^{\prime}$, we have, 
\[
\int\exp\left\{ -\left(\grave{\boldsymbol{\theta}}_{1}-\boldsymbol{\theta}_{0}^{\prime}\right)^{\top}\left(\check{\varOmega}_{1}\right)^{-1}\left(\grave{\boldsymbol{\theta}}_{1}-\boldsymbol{\theta}_{0}^{\prime}\right)\right\} \rho\left(\boldsymbol{\theta}_{0}^{\prime}\right)d\boldsymbol{\theta}_{0}^{\prime}.
\]
If we take a Lebesgue measure for $\rho$, the integrated value is
$\sqrt{\left(2\pi\left|\check{\varOmega}_{1}\right|\right)^{J}}$.
If $\rho=\mathcal{N}\left(0,W_{0}\right)$, then the inside of the
exponent is 
\begin{alignat*}{1}
 & \left(\grave{\boldsymbol{\theta}}_{1}-\boldsymbol{\theta}_{0}^{\prime}\right)^{\top}\left(\check{\varOmega}_{1}\right)^{-1}\left(\grave{\boldsymbol{\theta}}_{1}-\boldsymbol{\theta}_{0}^{\prime}\right)+\left(\boldsymbol{\theta}_{0}^{\prime}\right)^{\top}W_{0}^{-1}\left(\boldsymbol{\theta}_{0}^{\prime}\right)\\
= & \left(\boldsymbol{\theta}_{0}^{\prime}-\check{\boldsymbol{\theta}}_{0}\right)^{\top}\left(\check{\varSigma}_{0}\right)^{-1}\left(\boldsymbol{\theta}_{0}^{\prime}-\check{\boldsymbol{\theta}}_{0}\right)+\left(\grave{\boldsymbol{\theta}}_{1}\right)^{\top}\left(\check{\varOmega}_{0}\right)^{-1}\left(\grave{\boldsymbol{\theta}}_{1}\right),
\end{alignat*}
where, 
\begin{alignat*}{1}
\check{\boldsymbol{\theta}}_{0} & =\left(\check{\varOmega}_{1}^{-1}+W_{0}^{-1}\right)^{-1}\check{\varOmega}_{1}^{-1}\grave{\boldsymbol{\theta}}_{1},\\
\check{\varSigma}_{0} & =\left(\check{\varOmega}_{1}^{-1}+W_{0}^{-1}\right)^{-1},\\
\check{\varOmega}_{0} & =\check{\varOmega}_{1}+W_{0}.
\end{alignat*}
Therefore, the integrated value is 
\[
\exp\left\{ -\left(\grave{\boldsymbol{\theta}}_{1}\right)^{\top}\left(\check{\varOmega}_{0}\right)^{-1}\left(\grave{\boldsymbol{\theta}}_{1}\right)\right\} \sqrt{\left(2\pi\left|\check{\varSigma}_{0}\right|\right)^{J}}.
\]
Therefore, the denominator of eq.~\eqref{eq:Apnd1} can be expressed
as 
\begin{align}
 & \int_{\left(\mathbb{R}^{J}\right)^{t}}\beta_{t}\left(\boldsymbol{\theta}^{\prime},y\right)d\mu_{\theta}^{\rho}\left(\boldsymbol{\theta}_{0}^{\prime},\cdots,\boldsymbol{\theta}_{t}^{\prime}\right)\nonumber \\
= & \int_{\left(\mathbb{R}^{J}\right)^{t}}\beta_{t}\left(\boldsymbol{\theta}^{\prime},y\right)d\mu_{\theta}^{\mathsf{m}}\left(\boldsymbol{\theta}_{0}^{\prime},\cdots,\boldsymbol{\theta}_{t+1}^{\prime}\right)\times\frac{\exp\left\{ -\left(\grave{\boldsymbol{\theta}}_{1}\right)^{\top}\left(\check{\varOmega}_{0}\right)^{-1}\left(\grave{\boldsymbol{\theta}}_{1}\right)\right\} \sqrt{\left(2\pi\left|\check{\varSigma}_{0}\right|\right)^{J}}}{\sqrt{\left(2\pi\left|\check{\varOmega}_{1}\right|\right)^{J}}}.\label{eq:KFdenom}
\end{align}

Finally, from eq.~\eqref{eq:KFnume} and eq.~\eqref{eq:KFdenom},
the corollary holds. Q.E.D.

\subsection{Proof of Lemma \ref{lem:BRD_bound}\label{App:lemma3}}

For the DLM model, eqs.~\eqref{eq:DLMa}, \eqref{eq:DLMb}, and \eqref{eq:DLMc},
set the two priors on $v_{0}$ as $\rho\left(v_{0}\right),\varrho\left(v_{0}\right)$
and denote the cylindrical measure (joint pdf) obtained from those
priors, $\left\{ y_{s}\right\} _{s=1}^{t}$, as 
\begin{alignat*}{1}
\mu_{Y}^{\rho}\left(y_{1},\cdots,y_{t}\right) & \triangleq\int_{\left(0,\infty\right]}\mu_{Y}\left(\left.y_{1},\cdots,y_{t}\right|v_{0}\right)\rho\left(v_{0}\right)dv_{0}\\
\mu_{Y}^{\varrho}\left(y_{1},\cdots,y_{t}\right) & \triangleq\int_{\left(0,\infty\right]}\mu_{Y}\left(\left.y_{1},\cdots,y_{t}\right|v_{0}\right)\varrho\left(v_{0}\right)dv_{0}.
\end{alignat*}
Then, the Bayes risk difference is written as 
\begin{alignat*}{1}
 & B\left(\rho,\hat{p}_{t}^{\rho}\right)-B\left(\rho,\hat{p}_{t}^{\varrho}\right)\\
= & \int_{Y}\log\frac{\mu_{Y}^{\rho}\left(y_{1},\cdots,y_{t+1}\right)}{\mu_{Y}^{\varrho}\left(y_{1},\cdots,y_{t+1}\right)}d\mu_{Y}^{\rho}\left(y_{1},\cdots,y_{t+1}\right)\\
 & -\int_{Y}\log\frac{\mu_{Y}^{\rho}\left(y_{1},\cdots,y_{t}\right)}{\mu_{Y}^{\varrho}\left(y_{1},\cdots,y_{t}\right)}d\mu_{Y}^{\rho}\left(y_{1},\cdots,y_{t}\right).
\end{alignat*}
The second term in the RHS is, from the data processing inequality,
\[
\int_{Y}\log\frac{\mu_{Y}^{\rho}\left(y_{1},\cdots,y_{t}\right)}{\mu_{Y}^{\varrho}\left(y_{1},\cdots,y_{t}\right)}d\mu_{Y}^{\rho}\left(y_{1},\cdots,y_{t}\right)\geqq\int_{U}\log\frac{\mu_{U}^{\rho}\left(u_{1},\cdots,u_{t}\right)}{\mu_{U}^{\varrho}\left(u_{1},\cdots,u_{t}\right)}d\mu_{U}^{\rho}\left(u_{1},\cdots,u_{t}\right).
\]
Here, we have $\mu_{U}^{\rho}\left(u_{1},\cdots,u_{t}\right)=\mu_{Y}^{\rho}\left(f\left(y_{1}\right),\cdots,f\left(y_{t}\right)\right)$.
This is because $f$ is a measurable function, where there exists
a Markov kernel, $\left(y_{1},\cdots,y_{t}\right)\rightarrow\left(u_{1},\cdots,u_{t}\right)$,
which, using $\kappa\left(\mathbf{u}\left|\mathbf{y}\right.\right)$,
can be written as 
\[
\mu_{U}^{\rho}\left(u_{1},\cdots,u_{t}\right)=\int\cdots\int\kappa\left(\mathbf{u}\left|\mathbf{y}\right.\right)\mu_{Y}^{\rho}\left(y_{1},\cdots,y_{t}\right)dy_{1}dy_{2}\cdots dy_{t}.
\]
Then, 
\begin{alignat*}{1}
 & \int_{Y}\log\frac{\mu_{Y}^{\rho}\left(y_{1},\cdots,y_{t}\right)}{\mu_{Y}^{\varrho}\left(y_{1},\cdots,y_{t}\right)}d\mu_{Y}^{\rho}\left(y_{1},\cdots,y_{t}\right)\\
= & \int_{U}\int_{Y}\log\frac{\kappa\left(\mathbf{u}\left|\mathbf{y}\right.\right)\mu_{Y}^{\rho}\left(y_{1},\cdots,y_{t}\right)}{\kappa\left(\mathbf{u}\left|\mathbf{y}\right.\right)\mu_{Y}^{\varrho}\left(y_{1},\cdots,y_{t}\right)}\kappa\left(\mathbf{u}\left|\mathbf{y}\right.\right)d\mu_{Y}^{\rho}\left(y_{1},\cdots,y_{t}\right)\\
\geqq & \int_{U}\log\frac{\int_{Y}\kappa\left(\mathbf{u}\left|\mathbf{y}\right.\right)d\mu_{Y}^{\rho}\left(y_{1},\cdots,y_{t}\right)}{\int_{Y}\kappa\left(\mathbf{u}\left|\mathbf{y}\right.\right)d\mu_{Y}^{\varrho}\left(y_{1},\cdots,y_{t}\right)}\int_{Y}\kappa\left(\mathbf{u}\left|\mathbf{y}\right.\right)d\mu_{Y}^{\rho}\left(y_{1},\cdots,y_{t}\right)\\
= & \int_{U}\log\frac{\mu_{U}^{\rho}\left(u_{1},\cdots,u_{t}\right)}{\mu_{U}^{\varrho}\left(u_{1},\cdots,u_{t}\right)}d\mu_{U}^{\rho}\left(u_{1},\cdots,u_{t}\right)
\end{alignat*}
where the inequality is from the log-sum inequality is applied to
Lebesgue integrals: when a non-negative sequence, $\left\{ a_{i}\right\} _{i=1}^{\infty},\left\{ b_{i}\right\} _{i=1}^{\infty}$,
satisfies $\sum_{i=1}^{\infty}a_{i}<\infty,\sum_{i=1}^{\infty}b_{i}<\infty$,
\[
\sum_{i=1}^{\infty}\left(a_{i}\log\frac{a_{i}}{b_{i}}\right)\geqq\left(\sum_{i=1}^{\infty}a_{i}\right)\log\frac{\sum_{i=1}^{\infty}a_{i}}{\sum_{i=1}^{\infty}b_{i}}
\]
holds. Therefore, 
\begin{alignat}{1}
B\left(\rho,\hat{p}_{t}^{\rho}\right)-B\left(\rho,\hat{p}_{t}^{\varrho}\right) & \leqq\int_{Y}\log\frac{\mu_{Y}^{\rho}\left(y_{1},\cdots,y_{t+1}\right)}{\mu_{Y}^{\varrho}\left(y_{1},\cdots,y_{t+1}\right)}d\mu_{Y}^{\rho}\left(y_{1},\cdots,y_{t+1}\right)-\int_{U}\log\frac{\mu_{U}^{\rho}\left(u_{1},\cdots,u_{t}\right)}{\mu_{U}^{\varrho}\left(u_{1},\cdots,u_{t}\right)}d\mu_{U}^{\rho}\left(u_{1},\cdots,u_{t}\right).\label{eq:BayesRiskBound}
\end{alignat}

Next, for the first term in the RHS in eq.~\eqref{eq:BayesRiskBound},
from Jensen's inequality, 
\begin{alignat*}{1}
-\log\frac{\int_{\left(0,\infty\right]}\mu_{Y}\left(\left.y_{1},\cdots,y_{t+1}\right|v_{0}\right)\frac{\varrho\left(v_{0}\right)}{\rho\left(v_{0}\right)}\rho\left(v_{0}\right)dv_{0}}{\int_{\left(0,\infty\right]}\mu_{Y}\left(\left.y_{1},\cdots,y_{t+1}\right|v_{0}\right)\rho\left(v_{0}\right)dv_{0}} & =-\log\int_{\left(0,\infty\right]}\frac{\varrho\left(v_{0}\right)}{\rho\left(v_{0}\right)}\rho\left(\left.v_{0}\right|y_{1},\cdots,y_{t+1}\right)dv_{0}\\
 & \leqq-\int_{\left(0,\infty\right]}\log\left(\frac{\varrho\left(v_{0}\right)}{\rho\left(v_{0}\right)}\right)\rho\left(\left.v_{0}\right|y_{1},\cdots,y_{t+1}\right)dv_{0}.
\end{alignat*}
Therefore, 
\begin{alignat*}{1}
\int_{Y}\log\frac{\mu_{Y}^{\rho}\left(y_{1},\cdots,y_{t+1}\right)}{\mu_{Y}^{\varrho}\left(y_{1},\cdots,y_{t+1}\right)}d\mu_{Y}^{\rho}\left(y_{1},\cdots,y_{t+1}\right) & \geqq\int_{Y}\int_{\left(0,\infty\right]}\log\left(\frac{\rho\left(v_{0}\right)}{\varrho\left(v_{0}\right)}\right)\rho\left(\left.v_{0}\right|y_{1},\cdots,y_{t+1}\right)dv_{0}d\mu_{Y}^{\rho}\left(y_{1},\cdots,y_{t+1}\right)\\
 & =\int_{\left(0,\infty\right]}\log\left(\frac{\rho\left(v_{0}\right)}{\varrho\left(v_{0}\right)}\right)\rho\left(v_{0}\right)dv_{0}.
\end{alignat*}
Similarly for the second term in the RHS of eq.~\eqref{eq:BayesRiskBound},
we have, 
\begin{alignat*}{1}
\log\frac{\int_{\left(0,\infty\right]}\mu_{U}^{\rho}\left(\left.u_{1},\cdots,u_{t}\right|\tilde{v}_{0}\right)\frac{\rho\left(\tilde{v}_{0}\right)}{\varrho\left(\tilde{v}_{0}\right)}\varrho\left(\tilde{v}_{0}\right)d\tilde{v}_{0}}{\int_{\left(0,\infty\right]}\mu_{U}^{\varrho}\left(\left.u_{1},\cdots,u_{t}\right|\tilde{v}_{0}\right)\varrho\left(\tilde{v}_{0}\right)d\tilde{v}_{0}} & =\log\int_{\left(0,\infty\right]}\frac{\rho\left(\tilde{v}_{0}\right)}{\varrho\left(\tilde{v}_{0}\right)}\varrho\left(\left.\tilde{v}_{0}\right|u_{1},\cdots,u_{t}\right)d\tilde{v}_{0}\\
 & \geqq\int_{\left(0,\infty\right]}\log\left(\frac{\rho\left(\tilde{v}_{0}\right)}{\varrho\left(\tilde{v}_{0}\right)}\right)\varrho\left(\left.\tilde{v}_{0}\right|u_{1},\cdots,u_{t}\right)d\tilde{v}_{0}
\end{alignat*}
from which we have 
\[
\int_{U}\log\frac{\mu_{U}^{\rho}\left(u_{1},\cdots,u_{t}\right)}{\mu_{U}^{\varrho}\left(u_{1},\cdots,u_{t}\right)}d\mu_{U}^{\rho}\left(u_{1},\cdots,u_{t}\right)\geqq\int_{U}\int_{\left(0,\infty\right]}\log\left(\frac{\rho\left(\tilde{v}_{0}\right)}{\varrho\left(\tilde{v}_{0}\right)}\right)\varrho\left(\left.\tilde{v}_{0}\right|u_{1},\cdots,u_{t}\right)d\tilde{v}_{0}d\mu_{U}^{\rho}\left(u_{1},\cdots,u_{t}\right).
\]
Therefore, the RHS of eq.~\eqref{eq:BayesRiskBound} is bounded by
\[
\int_{U}\int_{\left(0,\infty\right]}\int_{\left(0,\infty\right]}\left\{ \log\frac{\rho\left(v_{0}\right)}{\varrho\left(v_{0}\right)}-\log\frac{\rho\left(\tilde{v}_{0}\right)}{\varrho\left(\tilde{v}_{0}\right)}\right\} \rho\left(v_{0}\right)dv_{0}\varrho\left(\left.\tilde{v}_{0}\right|u_{1},\cdots,u_{t}\right)d\tilde{v}_{0}d\mu_{U}\left(u_{1},\cdots,u_{t}\right)
\]
from above.

\section{Simulation study\label{sec:sim}}

To exemplify the theoretical results, we will use a simple, yet pertinent,
simulation study. Specifically, the simulation study is designed to
induce bias and dependence amongst agent forecasts, a setting that
the theoretical results highlight as the source of performative gains
in BPS. This setting is also realistic, in the sense that these characteristics
are observed in real world situations. The simulation study is also
relatively small sample (300 at most), which exemplifies the finite
sample properties and mirrors sample size seen in practice. To measure
predictive performances, we utilize the mean squared forecast error
(MSFE) and log predictive density ratio (LPDR) and compare dynamic
BPS against other ensemble methods.

\subsection{Simulation set up}

We construct a simulation study that captures the characteristics
encountered in real empirical applications; namely dependence amongst
agents and misspecification. First, the data generating process for
the target, $y_{t}$, is generated as follows: 
\begin{align*}
y_{t} & =\theta_{0,t}+\sum_{i=1}^{4}\theta_{ti}x_{ti}+\textrm{exp}(g_{t}/2)\nu,\quad\nu\sim N(0,1),\\
g_{t} & =g_{t-1}+\eta,\quad\eta\sim N(0,\sigma^{2}),\\
\theta_{0,t} & =\theta_{0,t-1}+\omega_{\theta_{0}},\quad\omega_{\theta_{0}}\sim N(0,\sigma^{2}),\\
\btheta_{t} & =\btheta_{t-1}+\bomega_{\btheta},\quad\bomega_{\btheta}\sim N(0,\sigma^{2}I),
\end{align*}
where the time varying parameters follow a random walk and the observation
error has stochastic volatility. We initialize $\btheta_{0}=1$ and
$\theta_{0}=0$ and discard the first 50 samples to allow random
starting points, and initialize $\textrm{exp}(g_{0}/2)=0.1$. The
covariates $x_{1{:}4}$ are generated as i.i.d. samples from $N(0,\sigma^{2})$.
For the simulation, $\sigma=0.01$.

At all points, only $\{y,x_{1{:}3}\}$ are observable by the agents
and $x_{4}$ is omitted, making all models misspecified. We construct
two agents, $\{\mA_{1},\mA_{2}\}$, with each only observing either
$\{x_{1},x_{3}\}$ or $\{x_{2},x_{3}\}$, thus allowing for dependencies.
Both agents forecast $y_{t}$ using a standard conjugate random walk
DLM ~\citep[Section 4,][]{WestHarrison1997book2}. Prior specifications
for the DLM state vector and discount volatility model in each of
the two agent models is based on $\btheta_{0}|v_{0}\sim N(\zero,(v_{0}/s_{0})\I)$
and $1/v_{0}\sim G(n_{0}/2,n_{0}s_{0}/2)$ with $n_{0}=2,s_{0}=0.01$.
Discount factors are set to $(\beta,\delta)=(0.99,0.95)$.

The BPS synthesis function follows eq.~\eqref{DLM}, with priors
$\theta_{0}|v_{0}\sim N(0,v_{0}/s_{0})$, $\btheta_{0}|v_{0}\sim N(\m_{0},(v_{0}/s_{0})\sigma^{2})$
with $\m_{0}=(0,\one'/J)'$ and $1/v_{0}\sim G(n_{0}/2,n_{0}s_{0}/2)$
with $n_{0}=10,s_{0}=0.002$. The discount factors are set to $(\beta,\delta)=(0.95,0.99)$.
Priors and discount factors are identical to \citet{mcalinn2019dynamic}.

For the study, we generate 350 samples from the data generating process.
The first 50 are used to sequentially estimate the agent models, with
the latter 25 of the 50 are used to simultaneously calibrate the BPS
synthesis function, as well as the other ensemble methods. Forecasts
are done sequentially over the remaining 300 samples, where agent
models, BPS, and the other ensemble methods are recalibrated for each
$t$ after observing new data and forecasts. We compare the mean squared
forecast error (MSFE) and the log predictive density ratio (LPDR),
evaluated at $t=\{100,200,300\}$. The log predictive density ratios
(LPDR) for each $t$ is 
\begin{align*}
\mathrm{LPDR}_{\seq1t}=\sum_{i=\seq1t}\mathrm{log}\{p_{*}(y_{t+1}|y_{1{:}t})/p_{\mathrm{BPS}}(y_{t+1}|y_{1{:}t})\}
\end{align*}
where $p_{*}(y_{t+1}|y_{1{:}t})$ is the predictive density of the
model being compared with. Comparing both the MSFE and LPDR provides
a more holistic assessment of the predictive performance. We repeat
the simulation 100 times and report the average.

\subsection{Simulation results}

Comparing MSFE and LPDR (Table~\ref{table:mse}), we find significant
improvements in favor of BPS, with BPS cutting down the MSFE by at
least half compared to the competing strategies. Additionally, the
improvements increase with $t$, showing how BPS dynamically adapts
to the data, accumulating improvements over time. The results are
consistent with the LPDR, also improving as $t$ increases.

\begin{table}[t!]
\centering \caption{Predictive evaluation using mean squared forecast error (MSFE) and
log predictive density ratio (LPDR) for equal weight averaging (EW),
Bayesian model averaging (BMA), and Bayesian predictive synthesis
(BPS), averaged over the 100 simulations. MSFE and LPDR are evaluated
at $t=\{100,200,300\}$. For MSFE, we report the ratio, MSFE$_{*}$/MSFE$_{BPS}$,
where {*} denotes the method compared against.}
\begin{tabular}{rrrrr}
\multicolumn{1}{l}{} &  & EW  & BMA  & BPS \tabularnewline
\hline 
\multicolumn{1}{r}{$t=100$} & MSFE$_{1:t}$  & 1.4738  & 1.4421  & 1.0000 \tabularnewline
 & \multicolumn{1}{r}{LPDR$_{1:t}$} & -16.92  & -16.21  & -- \tabularnewline
\multicolumn{1}{r}{$t=200$} & MSFE$_{1:t}$  & 1.5774  & 1.5584  & 1.0000 \tabularnewline
 & \multicolumn{1}{r}{LPDR$_{1:t}$} & -39.11  & -37.66  & -- \tabularnewline
\multicolumn{1}{r}{$t=300$} & MSFE$_{1:t}$  & 1.7550  & 1.7427  & 1.0000 \tabularnewline
 & \multicolumn{1}{r}{LPDR$_{1:t}$} & -69.53  & -67.30  & -- \tabularnewline
\end{tabular}\label{table:mse} 
\end{table}

The simulation results confirm and highlight the results from the
theoretical analysis in Section~\ref{sec:minmax}: dynamic BPS improves
forecasts over linear combinations. 
\end{document}